\documentclass[a4paper,reqno]{amsart}

\usepackage[T1]{fontenc}
\usepackage{lmodern}
\usepackage{microtype}
\usepackage{amssymb,amsmath,amsthm,amsfonts,mathtools,mathrsfs}
\usepackage{graphicx}
\usepackage[all]{xy}
\SelectTips{cm}{}
\usepackage{caption}
\usepackage{pgf,tikz}
\usetikzlibrary{arrows,arrows.meta}
\usepackage{tikz-cd}
\usepackage{adjustbox}
\usepackage{verbatim}
\usepackage{enumitem}
\usepackage{booktabs}
\usepackage{needspace}
\usepackage{xcolor}
\usepackage[a4paper, margin=1in]{geometry}

\usepackage[colorlinks=true,linkcolor=blue!55!black,citecolor=blue!55!black,
urlcolor=blue!55!black]{hyperref}
\hypersetup{
	pdftitle={Ideal filtrations and stabilization in alternating HRS tilting},
	pdfauthor={Zhe Han and Ping He},
	pdfsubject={Alternating HRS tilting, HRS ideal filtrations, and stabilization},
	pdfkeywords={HRS tilt, torsion pair, realization functor, HRS ideal filtration,
		HRS height, two-term silting, square-zero extension}
}

\setlist[enumerate]{label=(\roman*),leftmargin=2.2em}
\setlist[itemize]{leftmargin=2.0em}

\theoremstyle{plain}
\newtheorem{thm}{Theorem}[section]
\newtheorem{theorem}[thm]{Theorem}

\newtheorem{corollary}[thm]{Corollary}

\newtheorem{lemma}[thm]{Lemma}

\newtheorem{proposition}[thm]{Proposition}

\newtheorem{thmx}{Theorem}

\theoremstyle{definition}

\newtheorem{definition}[thm]{Definition}
\newtheorem{exm}{Example}

\theoremstyle{remark}

\newtheorem{remark}[thm]{Remark}
\newtheorem*{remark*}{Remark}

\renewcommand{\mod}{\operatorname{mod}}
\newcommand{\Ann}{\operatorname{Ann}}
\newcommand{\proj}{\operatorname{proj}}
\newcommand{\thick}{\operatorname{thick}}
\newcommand{\Hom}{\operatorname{Hom}}
\newcommand{\End}{\operatorname{End}}
\newcommand{\Ext}{\operatorname{Ext}}
\newcommand{\rad}{\operatorname{rad}}
\newcommand{\add}{\operatorname{add}}
\renewcommand{\Im}{\operatorname{Im}}
\newcommand{\Coker}{\operatorname{Coker}}
\newcommand{\Fac}{\operatorname{Fac}}
\newcommand{\Sub}{\operatorname{Sub}}
\newcommand{\Adm}{\operatorname{Adm}}

\newcommand{\T}{\mathcal{T}}
\newcommand{\F}{\mathcal{F}}
\renewcommand{\H}{\mathcal{H}}
\newcommand{\B}{\mathcal{B}}
\newcommand{\A}{\mathcal{A}}
\newcommand{\D}{\mathcal{D}}

\newcommand{\C}{\mathcal{C}}
\renewcommand{\P}{\mathbf{P}}
\newcommand{\Q}{\mathbf{Q}}
\newcommand{\V}{\mathbf{v}}
\newcommand{\w}{\mathbf{w}}
\newcommand{\uu}{\mathbf{u}}

\newcommand{\Db}{\mathrm D^{\mathrm b}}
\newcommand{\Kb}{\mathrm K^{\mathrm b}}
\newcommand{\len}{\ell}

\title[Ideal filtrations and stabilization in HRS tilting]
{Ideal filtrations and stabilization in alternating HRS tilting}

\author{Zhe Han}
\address{Z.H.: School of Mathematics and Statistics, Henan University, 475004 Kaifeng, China}
\email{zhehan@vip.henu.edu.cn}

\author{Ping He$^*$}
\address{P.H.: Beijing Institute of Mathematical Sciences and Applications, 101408 Beijing, China}
\email{pinghe@bimsa.cn}

\thanks{$^*$ Corresponding author}
\subjclass[2020]{Primary 18G80; Secondary 16G10, 18E40}
\keywords{HRS tilt, torsion pair, realization functor, derived equivalence, HRS ideal filtration, HRS height, two-term silting. }
\begin{document}
\begin{abstract}
We prove that the realization functors obtained by alternating forward
and backward HRS tilts of any torsion pair over a finite-dimensional
algebra are equivalences from some step onward. The number of steps
is bounded by the lengths of
the indecomposable projective modules. The even hearts define an
increasing filtration of ideals contained in a square-zero ideal.
For a torsion pair induced by a two-term silting complex, we identify
the successive layers with Buan--Zhou kernel ideals and compute them
by reduction modulo the successive ideals. We also describe the largest quotient over
which the complex becomes tilting and determine the exact
stabilization index from the two parity heights. We construct
algebras for which
the HRS height is arbitrarily large, and compute their ideal filtrations.
\end{abstract}
\maketitle

\section{Introduction}
Happel--Reiten--Smal{\o} tilting associates an abelian heart to a torsion
pair \cite{HRS}. The corresponding realization functor need not be a
derived equivalence; criteria for being an equivalence are given in
\cite{CHZ,pv,sr}. Psaroudakis and Vit\'oria study realization
functors from \cite{BBD} in a wider derived Morita framework and relate their
equivalence property to tilting and cotilting $t$-structures \cite{pv}.
For a general HRS tilt, the realization functor is an equivalence
if and only if the exact fully faithful embedding of the second,
backward HRS heart into the original category is essentially surjective
\cite{CHZ}. Recent applications
of this criterion to particular torsion pairs and silting mutation
appear in \cite{Pavon}.
For finite-dimensional hereditary algebras, Stanley and van
Roosmalen characterize derived-equivalent bounded hearts by a
Serre-functor condition on their aisles \cite{sr}.
Han and He study the corresponding criterion for HRS tilts in terms
of effaceable torsion pairs \cite{HanHeSerre}.

In this paper we study the realization functors arising from repeated
HRS tilts. Starting with a torsion pair $(\T,\F)$ in $\mod\Lambda$,
we alternate forward and backward tilts at the induced torsion pairs.
Each new heart is formed in the bounded derived category of the
preceding heart. Write
\[
\A^{(0)}=\mod\Lambda,\qquad
G^{(m)}:\Db(\A^{(m)})\longrightarrow\Db(\A^{(m-1)})\quad(m\ge1)
\]
for these hearts and realization functors. We first prove that the
realization functors are equivalences from some step onward. This
holds for every torsion pair, with no functorial finiteness assumption.

\begin{thmx}\label{thm:intro-main}
Let $\Lambda$ be a finite-dimensional algebra over any field. If
$P_i\twoheadrightarrow S_i$, $1\le i\le r$, are the projective
covers of its pairwise nonisomorphic simple modules, put
$d_\Lambda=\sum_{i=1}^r\bigl(\len_\Lambda(P_i)-1\bigr).$
For every torsion pair there is $0\le N\le d_\Lambda$ such that the functor
$G^{(m)}$ is a triangle equivalence for all $m\ge2N+1$.
\end{thmx}

Over $\Lambda$, the embedded images of the even hearts have the form
$\C_n=\mod(\Lambda/I_n)$ as embedded subcategories. We call
\[
0=I_0\subseteq I_1\subseteq\cdots\subseteq J_\T,
\qquad J_\T=\Ann_\Lambda\T\cap\Ann_\Lambda\F,
\]
the \emph{HRS ideal filtration}. The common annihilator satisfies
$J_\T^2=0$. Noohi's explicit model \cite{Noohi} gives an intrinsic
description of $I_1=I_\Lambda(\T)$ using only submodules, quotients,
and membership in the torsion pair. On writing $\Lambda_n=\Lambda/I_n$
and transporting the pair to $(\T_n,\F_n)$, the same construction gives
\[
I_{n+1}/I_n=I_{\Lambda_n}(\T_n).
\]
The \emph{HRS height} $h_\Lambda(\T)=\min\{n:I_{n+1}=I_n\}$ satisfies
$h_\Lambda(\T)\le
\min\{d_\Lambda,\len_{\Lambda^{\mathrm{op}}\otimes_k\Lambda}(J_\T)\},$
and every realization functor from step $2h_\Lambda(\T)+1$ onward
is an equivalence. Write $I_\infty$ for the stable value of the filtration.
It need not equal $J_\T$, although every $I_n$ is contained in $J_\T$.

Isomorphism classes of basic two-term silting complexes correspond
bijectively to those of basic support $\tau$-tilting modules and to
functorially finite torsion classes \cite{AIR}. For a two-term silting
complex $\P$, the first HRS heart is a module category over
$\End(\P)$ \cite{KY}. Its
one-step realization is an equivalence exactly when $\P$ is tilting
\cite[Proposition~4.7]{CHZ}; compare \cite[Corollary~5.2]{pv}.

For a two-term silting complex $\P$, Buan and Zhou construct an
algebra epimorphism $\Phi_\P$ whose kernel is generated by negative
morphisms \cite{BZ}. Using the dg description of 
\cite{XYZ}, we show that restriction along $\Phi_\P$ identifies
modules over its target algebra with the essential image of the
second HRS heart in $\mod\Lambda$. Hence $\ker\Phi_\P$ is the first
HRS ideal. Applying
this identification to the reductions of $\P$ gives the following
formulas.

\begin{thmx}\label{thm:intro-calculus}
Let $\P=[P^{-1}\xrightarrow d P^0]$ be two-term silting
and let $(I_n)$ be the HRS ideal filtration of its torsion pair.
Put $\P_n=\P\otimes_\Lambda\Lambda_n$.
\begin{enumerate}
\item The complex $\P_n$ is silting and induces the transported pair.
The ideal $I_{n+1}/I_n$ is generated by the matrix coefficients of
\[
\{H: P_n^0\to P_n^{-1}\mid Hd_n=d_nH=0\},
\]
where $d_n$ is the reduced differential and
$P_n^{-1}=P^{-1}/P^{-1}I_n$,
$ P_n^0= P^0/ P^0I_n$.
\item With $R_n=H^{-1}(\P)\cap P^{-1}I_n$, one has
$H^{-1}(\P_n)=H^{-1}(\P)/R_n$ and
\[
h_\Lambda(\T(\P))=
\min\{n:\Hom_\Lambda(H^0(\P),H^{-1}(\P)/R_n)=0\}.
\]
The quotient $\Lambda/I_\infty$ is the largest quotient over which
the reduction of $\P$ is tilting.
\item Let $h_{\mathrm{ev}}$ and $h_{\mathrm{odd}}$ be the HRS
heights of the two parity towers. The least $s\ge1$ for which every
$G^{(m)}$, $m\ge s$, is an equivalence is
\[
\sigma_\Lambda(\T(\P))=h_{\mathrm{ev}}+h_{\mathrm{odd}}+1,
\qquad h_{\mathrm{odd}}\le h_{\mathrm{ev}}\le h_{\mathrm{odd}}+1.
\]
\end{enumerate}
\end{thmx}

Since every map $H^0(\P)\to H^{-1}(\P)$ has image in $R_1$,
its composite with $H^{-1}(\P)\twoheadrightarrow
 H^{-1}(\P)/R_n$ is zero for $n\ge1$. The connecting
homomorphism therefore gives
\[
\begin{aligned}
\Hom_\Lambda(H^0(\P),H^{-1}(\P)/R_n)
\simeq\ker\bigl[\Ext^1_\Lambda(H^0(\P),R_n)
\longrightarrow\Ext^1_\Lambda(H^0(\P),H^{-1}(\P))\bigr]
\quad(n\ge1).
\end{aligned}
\]
We use this description to obtain bounds in terms of the lengths of
$H^{-1}(\P)$ and $\Hom_\Lambda(H^0(\P),\Lambda)$,
and a lifting criterion for
height at most one. In particular, every two-term silting complex over a
radical-square-zero algebra has height at most one.
Over a hereditary algebra, the same height bound holds for every
torsion pair, without a functorial finiteness assumption. A 
tilted algebra of type $E_6$ has a complex of height two. Hence the supremum of
HRS heights of two-term silting complexes is not a derived invariant.

The number of simple modules and the Loewy length do not bound the
HRS height. Our construction starts from a nonzero finite-dimensional
vector space $V$ with a nilpotent endomorphism $\phi$. A square-zero
extension $\Lambda_\phi=B\ltimes\mathcal M_\phi$ of the Kronecker
algebra $B$, together with a two-term silting complex $\P_\phi$,
realizes the kernel filtration of $\phi$ as an HRS ideal filtration:
\[
I_m=\mathcal M_\phi(\ker\phi^m),\qquad
h_{\Lambda_\phi}(\T(\P_\phi))=\nu(\phi),
\]
where $\nu(\phi)$ is the nilpotency index. Here
$\mathcal M_\phi(U)$ is the subbimodule obtained from a
$\phi$-invariant subspace $U\subseteq V$ by the same construction
(Section~\ref{subsec:nilpotent-operator-families}).
Varying the Jordan type yields the following result.

\begin{thmx}
	\label{thm:intro-examples}
	Let $k$ be a field and let $c_0\ge c_1\ge\cdots\ge c_{N-1}>0$ be
	integers, with $N\ge1$. 
	\begin{enumerate}
	\item There exist a finite-dimensional basic
	$k$-algebra $\Lambda$ with two simple modules and Loewy length three,
	and a basic two-term silting complex $\P$, such that
	\[
	\dim_k\Lambda=4+4\sum_{m=0}^{N-1}c_m,\qquad
	h_\Lambda(\T(\P))=N.
	\]
	\item For its HRS ideal filtration $\{I_m\}$, put $\Lambda_m=\Lambda/I_m$
	and $\P_m=\P\otimes_\Lambda\Lambda_m$. Then, for $0\le m<N$,
	\[
	\dim_k(I_{m+1}/I_m)=4c_m,\qquad
	\dim_k\Hom_{\Kb(\proj\Lambda_m)}(\P_m,\P_m[-1])=2c_m.
	\]
	\end{enumerate}
\end{thmx}

In particular, taking $c_0=\cdots=c_{N-1}=1$ gives algebras
$\Lambda_N$ of dimension $4N+4$ and complexes $\P_N$ of height $N$,
with every nonzero HRS ideal quotient of dimension four.
The filtration formula is Theorem~\ref{thm:nilpotent-operator-defect};
the layer dimensions and their realization follow from
Corollaries~\ref{cor:nilpotent-profile-dimensions}
and~\ref{cor:jordan-profile-realization}.

Section~\ref{sec:preliminaries} fixes the HRS conventions and recalls
	the second-heart embedding. Section~\ref{sec:iteration} proves finite
	stabilization using internal projective covers, and
	Section~\ref{sec:refined-invariants} develops the intrinsic ideal
	filtration, its stopping criterion and its general height bounds.
Section~\ref{sec:silting-hrs-ideals} establishes silting reduction,
	the matrix-coefficient formula and the exact parity-dependent
	stabilization index. Section~\ref{sec:kernel-lifting} proves the
	kernel-filtration theorem, and gives an example over a tilted algebra
	of type $E_6$ showing that HRS height is not a derived invariant. Finally,
	Section~\ref{sec:arbitrary-height} realizes arbitrary heights and
	prescribed layer dimensions by square-zero extensions associated
	with nilpotent operators.

\section{Alternating HRS tilts}\label{sec:preliminaries}
Throughout, $k$ is a field and $\Lambda$ is a finite-dimensional
basic $k$-algebra. 
We write $\mod\Lambda$ for the category of finite-dimensional right
$\Lambda$-modules. We use cohomological grading and write $\len_\A$
for composition length in a length category $\A$. The notations
$\add$, $\Fac$, and $\Sub$ denote additive closure, quotients of
finite direct sums, and submodules of finite direct sums, respectively.
Put $D=\Hom_k(-,k)$ and $\nu=D\Hom_\Lambda(-,\Lambda)$ on
projective complexes. A complex $\P\in\Kb(\proj\Lambda)$ is
\emph{silting} if $\Hom(\P,\P[i])=0$ for $i>0$ and
$\thick(\P)=\Kb(\proj\Lambda)$. It is \emph{tilting} if the
vanishing also holds for $i<0$. Here $\thick$ denotes closure
under shifts, cones, and direct summands.
A complex is called \emph{basic} if its indecomposable direct
summands are pairwise nonisomorphic.

All subcategories are full and closed under isomorphisms.
A \emph{torsion pair} in an abelian category $\A$ is a pair
$(\T,\F)$ of subcategories such that $\Hom_\A(\T,\F)=0$ and
every $X\in\A$ admits a short exact sequence
\[
0\longrightarrow tX\longrightarrow X\longrightarrow fX
\longrightarrow0,
\qquad tX\in\T,\quad fX\in\F.
\]
The torsion pair is \emph{split} if this sequence splits for
every $X\in\A$, equivalently if every $X$ is isomorphic to
$tX\oplus fX$. Equivalently, $\Ext^1_\A(F,T)=0$ for all
$F\in\F$ and $T\in\T$.
The forward HRS tilt of this pair is the heart
\[
\B=\{X\in\Db(\A):H^{-1}(X)\in\F,\ H^0(X)\in\T,
\ H^i(X)=0\ (i\ne-1,0)\}.
\]
It has the induced torsion pair $(\F[1],\T)$. The symbol $[1]$
here denotes the shift in $\Db(\A)$, so $\F[1]$ is a subcategory
of the abelian category $\B$.

We next work in $\Db(\B)$ and denote its shift by $\Sigma_\B$.
The associated backward tilt is
\[
\A'=\{Y\in\Db(\B):H^0_\B(Y)\in\T,\ H^1_\B(Y)\in\F[1],
\ H^i_\B(Y)=0\ (i\ne0,1)\}.
\]
Its induced torsion pair is $(\T,\Sigma_\B^{-1}(\F[1]))$.
The two tilts take place in different derived categories, so their
composite need not recover $\A$.

A realization functor for a bounded heart $\H\subseteq\D$ is a
$t$-exact triangle functor $\Db(\H)\to\D$ extending its inclusion.
All ambient categories used here are bounded derived categories
of abelian categories. Their standard enhancements provide realization
functors for the HRS hearts under consideration
\cite{KV,pv,CHZ,b}. 

We use the following result  to regard $\A'$
as an exact full abelian subcategory of $\A$.

\begin{proposition}[\cite{CHZ}, Propositions 3.1--3.3]
\label{prop:chz}
Let $G:\Db(\B)\to\Db(\A)$ be a realization functor.
Its restriction $G|_{\A'}:\A'\to\A$ is exact and fully faithful.
It identifies the induced torsion pair with $(\T,\F)$.
The following conditions are equivalent:
\begin{enumerate}
\item $G$ is an equivalence.
\item $G|_{\A'}$ is an equivalence.
\item The essential image of $G|_{\A'}$ is $\A$.
\end{enumerate}
In this case every realization functor for the associated backward
tilt is also an equivalence.
\end{proposition}

By \cite[Proposition~3.2]{CHZ}, the essential image of $G|_{\A'}$ depends only
on the torsion pair, independently of the chosen realization functor. Under the chosen identification of $G$ with the inclusion on $\B$,
we have
\[
G(T)=T,\qquad G\bigl(\Sigma_\B^{-1}(F[1])\bigr)=F
\quad(T\in\T,\ F\in\F).
\]
Thus the essential image of $G|_{\A'}$ contains both $\T$ and
$\F$. We iterate this construction as follows.

\begin{definition}\label{def:iterative}
Start with $\A^{(0)}=\A$ and the torsion pair $(\T,\F)$.
Form $\A^{(2n+1)}$ by the forward tilt of $\A^{(2n)}$.
Then form $\A^{(2n+2)}$ by the associated backward tilt, with
its induced torsion pair. Each heart is formed in the bounded
derived category of the preceding heart. Denote the chosen
realization functors by
\[
G^{(m)}:\Db(\A^{(m)})\longrightarrow\Db(\A^{(m-1)}),
\qquad m\ge1.
\]
\end{definition}

\section{Finite height and stabilization}\label{sec:iteration}
Let $\A$ be a length abelian category with simple objects
$S_1,\ldots,S_r$ up to isomorphism. Assume that each $S_i$ has a
finite-length projective cover $\pi_i:P_i\twoheadrightarrow S_i$.
An \emph{exact full abelian subcategory} of $\A$ is an abelian
subcategory whose inclusion is exact. We do not assume that it is
extension closed in $\A$. 

Let $\ell\colon\C\to\A$ be an exact fully faithful functor between abelian
categories.  Then $\ell$ reflects zero objects, monomorphisms, epimorphisms,
and exact sequences.  In particular, after identifying $\C$ with its
essential image, the kernel and cokernel in $\A$ of a morphism between
objects of $\C$ belong to $\C$ and agree with its kernel and cokernel in
$\C$.
Moreover, if $\C'\subseteq\C\subseteq\A$ are full abelian subcategories and both
inclusions into $\A$ are exact, then the inclusion
$\C'\subseteq\C$ is also exact. 

\begin{lemma}\label{lem:same-simples}
An exact full abelian subcategory $\C\subseteq\A$ containing all
$S_i$ is closed under subobjects and quotients in $\A$.
It is a length category with exactly the simple objects
$S_1,\ldots,S_r$.
\end{lemma}
\begin{proof}
Let $Y\subseteq X$ in $\A$, with $X\in\C$. A composition series
of $X/Y$ lifts to a chain
\[
X=X_0\supset X_1\supset\cdots\supset X_s=Y
\]
with simple successive factors. If $X_j\in\C$, then
$X_j\twoheadrightarrow X_j/X_{j+1}$ is a morphism in $\C$.
Since the inclusion is exact, its kernel $X_{j+1}$ belongs to $\C$.
Induction gives $Y\in\C$, and the cokernel of $Y\hookrightarrow X$
gives $X/Y\in\C$.

It follows that every composition series in $\A$ of an object of
$\C$ is also a composition series in $\C$. Thus $\C$ is a length
category with simple objects $S_1,\ldots,S_r$.
\end{proof}

\begin{proposition}\label{prop:internal-covers}\label{cor:proj-generator}
Let $\C\subseteq\A$ be an exact full abelian subcategory containing all $S_i$.
Every $S_i$ has a projective cover
$q_i^\C:Q_i^\C\twoheadrightarrow S_i$ in $\C$, with $Q_i^\C$
a quotient of $P_i$ in $\A$. The object
$Q^\C=\bigoplus_iQ_i^\C$ is a projective generator.
\end{proposition}
\begin{proof}
Since $S_i\in\C$ and quotient lengths are bounded by
$\len_\A(P_i)$, we may choose an epimorphism
$u:P_i\twoheadrightarrow Q$ with $Q\in\C$ of maximal length among
the quotients of $P_i$ lying in $\C$. Let
$f:P_i\to X$, where $X\in\C$, and put
$Y=\Im(u,f)\subseteq Q\oplus X$. By
Lemma~\ref{lem:same-simples}, $Y\in\C$. The projection
$p:Y\twoheadrightarrow Q$ is an epimorphism. Since $Y$ is a
quotient of $P_i$, maximality gives
\[
\len_\A(Q)\le\len_\A(Y)\le\len_\A(Q).
\]
Thus $p$ is an isomorphism and $f$ factors uniquely through $u$.
We obtain a natural isomorphism
\[
\Hom_\C(Q,-)\simeq\Hom_\A(P_i,-)|_\C.
\]
Since $P_i$ is projective and the inclusion $\C\hookrightarrow\A$
is exact, this functor is exact. Hence $Q$ is projective in $\C$.

The kernel of the projective cover $\pi_i$ is the unique maximal
subobject of $P_i$. Since $Q\ne0$, we have
$\ker u\subseteq\ker\pi_i$. Thus $\pi_i$ induces
$q:Q\twoheadrightarrow S_i$. Every proper subobject of $Q$ is
contained in $\ker q=(\ker\pi_i)/(\ker u)$. Thus $q$ is
an essential epimorphism, and hence a projective cover in $\C$.

Set $Q_i^\C=Q$. Since $\C$ is a length category and every simple
object is a quotient of one of the $Q_i^\C$, induction on length
shows that $Q^\C$ is a projective generator.
\end{proof}
\begin{corollary}
Every $X\in\C$ has a presentation $Q_1\to Q_0\to X\to0$
with $Q_0,Q_1\in\add Q^\C$.
\end{corollary}
\begin{proof}
For $X\ne0$ in $\C$, choose an exact sequence
$0\to K\to X\to S_i\to0$. Projectivity lifts $q_i^\C$ to
$Q_i^\C\to X$, and induction on length gives an epimorphism
$Q'\twoheadrightarrow K$ with $Q'\in\add Q^\C$. The induced map
$Q'\oplus Q_i^\C\to X$ is epic. Applying this construction to
its kernel gives the required presentation.
\end{proof}

\begin{proposition}\label{prop:strict-mu}
Set $\mu_\A(\C)=\sum_i\len_\A(Q_i^\C)$.
If $\C'\subseteq\C$ are exact full abelian subcategories containing
all $S_i$, then $\mu_\A(\C')\le\mu_\A(\C)$.
Equality holds if and only if $\C'=\C$.
\end{proposition}
\begin{proof}
Since $q_i^{\C'}:Q_i^{\C'}\twoheadrightarrow S_i$ is an
epimorphism in $\C$, projectivity gives
$f_i:Q_i^\C\to Q_i^{\C'}$ such that
$q_i^{\C'}f_i=q_i^\C$. The subobject $\Im f_i$ belongs to
$\C'$ by Lemma~\ref{lem:same-simples} and maps onto $S_i$.
Since $q_i^{\C'}$ is essential, $\Im f_i=Q_i^{\C'}$. Therefore
\[
\len_\A(Q_i^{\C'})\le\len_\A(Q_i^\C)
\quad\text{for every }i,
\]
which proves the inequality.

If $\mu_\A(\C')=\mu_\A(\C)$, each epimorphism $f_i$ has
source and target of equal length and is therefore an isomorphism.
The two categories have a common projective generator $Q$. For
$X\in\C$, take a presentation $Q_1\to Q_0\to X\to0$ with
$Q_0,Q_1\in\add Q$. Since $\C'$ is full and its inclusion is
exact, the cokernel of $Q_1\to Q_0$ belongs to $\C'$. Thus
$X\in\C'$ and $\C=\C'$. The converse is immediate.
\end{proof}

\begin{theorem}\label{thm:categorical-finite-height} 
Put $d_\A=\sum_{i=1}^r\bigl(\len_\A(P_i)-1\bigr)$.
Let $\A=\C_0\supsetneq\C_1\supsetneq\cdots\supsetneq\C_s$
be a strictly descending chain of exact full abelian subcategories
containing all $S_i$. Then $s\le d_\A$.
\end{theorem}
\begin{proof}
By Proposition~\ref{prop:strict-mu}, $\mu_\A$ decreases by at
least one at each inclusion. Since every $Q_i^{\C_s}$ is nonzero,
$\mu_\A(\C_s)\ge r$. It follows that
$s\le\mu_\A(\A)-\mu_\A(\C_s)
\le d_\A.$
\end{proof}
For a basic algebra $\Lambda$, one has
$d_\A=d_\Lambda=\len_\Lambda(\rad\Lambda)$ when $\A=\mod\Lambda$.
We apply Theorem~\ref{thm:categorical-finite-height} to the HRS iteration.
Proposition~\ref{prop:chz} gives exact fully faithful functors
\[
E_n=G^{(2n+1)}|_{\A^{(2n+2)}}:
\A^{(2n+2)}\longrightarrow\A^{(2n)}.
\]
The essential images of their composites in $\A$ form a chain
\begin{equation}\label{eq:even-chain}
\A=\C_0\supseteq\C_1\supseteq\C_2\supseteq\cdots.
\end{equation}
Since each composite identifies the induced torsion pair with
$(\T,\F)$, every $\C_n$ contains both classes. Each simple
object belongs to $\T$ or $\F$, as follows from its torsion
sequence. Thus every $\C_n$ contains all simple objects and has
a finite projective generator by
Proposition~\ref{cor:proj-generator}.

We next show that equality at one step forces equality at every
later step.

\begin{proposition}\label{prop:absorption}
If $\C_{n+1}=\C_n$, then $\C_j=\C_n$ for all $j\ge n$.
Every $G^{(m)}$ with $m\ge2n+1$ is then an equivalence.
\end{proposition}
\begin{proof}
Since $\C_{n+1}=\C_n$, the fully faithful functor $E_n$ is
an equivalence. Proposition~\ref{prop:chz} implies that
$G^{(2n+1)}$ and $G^{(2n+2)}$ are equivalences.

The functor $E_n$ identifies the torsion pairs on $\A^{(2n+2)}$
and $\A^{(2n)}$. Hence its derived equivalence restricts to an
equivalence of their forward HRS hearts,
$\A^{(2n+3)}\simeq\A^{(2n+1)}$. Transporting $G^{(2n+1)}$
along these equivalences gives an equivalence realization functor
for the forward tilt of $\A^{(2n+2)}$. By
\cite[Corollary 2.8]{CHZ}, the chosen realization $G^{(2n+3)}$
is also an equivalence. Proposition~\ref{prop:chz} then gives
$\C_{n+2}=\C_{n+1}$ and shows that $G^{(2n+4)}$ is an
equivalence. Induction completes the proof.
\end{proof}

\begin{theorem}\label{thm:categorical-extension}
For every torsion pair in $\A$, there is $0\le N\le d_\A$
such that $\C_n=\C_N$ for $n\ge N$.
Every $G^{(m)}$ with $m\ge2N+1$ is an equivalence.
In particular, the result applies to $\A=\mod\Lambda$ for any
finite-dimensional algebra $\Lambda$.
\end{theorem}
\begin{proof}
If $\C_{n+1}\ne\C_n$ for all $0\le n\le d_\A$, the first
$d_\A+1$ inclusions would be strict.
This contradicts Theorem~\ref{thm:categorical-finite-height}.
Thus equality occurs at some $N\le d_\A$.
Proposition~\ref{prop:absorption} gives the conclusion.
\end{proof}
\section{The intrinsic HRS ideal filtration}\label{sec:refined-invariants}
Throughout this section, $\A=\mod\Lambda$.
For a class $\mathcal X$ of modules, put
$\Ann_\Lambda\mathcal X=\bigcap_{X\in\mathcal X}\Ann_\Lambda X$.
We identify the subcategories $\C_n$ with module categories over
quotients of $\Lambda$. For an ideal $I\subseteq \Lambda$ the category $\mod(\Lambda/I)$ is embedded into $\mod\Lambda$ by restriction of scalars. 

\begin{proposition}\label{prop:quotient-description}
Every exact full abelian subcategory $\C\subseteq\mod\Lambda$
containing all simple modules has the form
\[
\C=\{X:XI_\C=0\}=\mod(\Lambda/I_\C),\qquad
I_\C=\Ann_\Lambda\C\subseteq\rad\Lambda.
\]
\end{proposition}
\begin{proof}
Put $I=I_\C$.
Among the finite intersections of annihilators of objects of $\C$,
choose one of minimal dimension. Since intersecting it with
$\Ann_\Lambda Y$ for any $Y\in\C$ cannot reduce its dimension,
it equals $I$.
Choose $X_1,\ldots,X_m\in\C$ realizing this finite intersection,
and put $X=\bigoplus_jX_j$. Then
\[
\Ann_\Lambda X=\bigcap_{j=1}^m\Ann_\Lambda X_j=I.
\]
Since $\Ann_{\Lambda/I}X=\Ann_\Lambda X/I=0$,
$X$ is faithful over $\Lambda/I$.
For a $k$-basis $x_1,\ldots,x_d$ of $X$, consider
\[
\Lambda/I\longrightarrow X^d,\qquad
a+I\longmapsto(x_1a,\ldots,x_da).
\]
Since $Xa=0$ implies $a\in I$, the displayed right
$\Lambda$-module map is injective.

By Lemma~\ref{lem:same-simples}, we have $\Lambda/I\in\C$.
Every finite $\Lambda/I$-module is a quotient of $(\Lambda/I)^r$
for some $r$, so the same lemma gives $\mod(\Lambda/I)\subseteq\C$.
The reverse inclusion follows from the definition of $I$.
Since $\C$ contains every simple module, $I\subseteq\rad\Lambda$.
Finally, $\Ann_\Lambda(\Lambda/I)=I$ gives uniqueness.
\end{proof}

Apply Proposition~\ref{prop:quotient-description} to
the chain \ref{eq:even-chain} and set
\begin{equation}\label{eq:ideal-chain}
I_n=\Ann_\Lambda\C_n,\qquad \Lambda_n=\Lambda/I_n,\qquad
0=I_0\subseteq I_1\subseteq I_2\subseteq\cdots.
\end{equation}
We call this chain the \emph{HRS ideal filtration}.
Restriction of scalars identifies $\C_n$ with $\mod\Lambda_n$.
After this identification, the transported classes satisfy
$\T_n=\T$ and $\F_n=\F$ as full subcategories of $\mod\Lambda$:
each class is already contained in $\C_n$.
In particular, both classes are annihilated by every $I_n$.

\begin{lemma}\label{lem:square-zero-annihilator}
Set $J_\T=\Ann_\Lambda\T\cap\Ann_\Lambda\F$.
Then the following hold.
\begin{enumerate}
\item[(1)] $J_\T^2=0$, $I_n\subseteq J_\T$ for all $n$;
\item[(2)]
		$\Ann_{\Lambda_n}\T_n\cap\Ann_{\Lambda_n}\F_n=J_\T/I_n.$
\end{enumerate}
\end{lemma}
\begin{proof}
(1) Take the torsion sequence
$0\to U\to\Lambda\to V\to0$ of the regular right module.
Since $V\in\F$, the equality $VJ_\T=0$ gives
$J_\T=\Lambda J_\T\subseteq U$.
Since $U\in\T$, we also have $UJ_\T=0$. Thus $J_\T^2\subseteq UJ_\T=0$.
Since $\T,\F\subseteq\C_n$, we have $I_n\subseteq J_\T$.

(2) For a $\Lambda_n$-module $X$, the coset $a+I_n$ annihilates $X$
if and only if $Xa=0$. Thus the two annihilators over $\Lambda_n$ are
$\Ann_\Lambda\T/I_n$ and $\Ann_\Lambda\F/I_n$.
Their intersection is $J_{\T}/I_n$.
\end{proof}

\subsection{An intrinsic construction}\label{sec:first-hrs-ideal}
We call $I_\Lambda(\T)=I_1$ the \emph{first HRS ideal} and write
$\C_\T=\C_1$.
Noohi's description of the second tilt gives the following
module-theoretic construction of $I_\Lambda(\T)$.

\begin{definition}\label{def:tau-admissible}
A module $X$ is \emph{$\T$-admissible} if there are submodules
$K_1\subseteq K_2\subseteq E$ and $L\subseteq E$ such that
\[
X\simeq K_2/K_1,\qquad K_2\cap L\in\F,\qquad
E/(K_1+L)\in\T.
\]
Write $\Adm_\T(\Lambda)$ for the full subcategory of these modules.
\end{definition}

\begin{proposition}\label{prop:noohi-admissible}
One has
$\Adm_\T(\Lambda)=\C_\T=\mod(\Lambda/I_\Lambda(\T))$.
In particular,
$I_\Lambda(\T)=\Ann_\Lambda\Adm_\T(\Lambda)$ is intrinsic
to the torsion pair.
\end{proposition}
\begin{proof}
By \cite[Section 8.1 and Proposition 8.1]{Noohi}, the second forward
HRS heart is described by diagrams
$[K_1\subseteq K_2\subseteq E\supseteq L]$ satisfying the conditions
of Definition~\ref{def:tau-admissible}.
In $\Db(\B)$, this heart has $\T$ in degree $-1$ and $\F[1]$
in degree $0$, so it is $\Sigma_\B\A'$.
After shifting by $\Sigma_\B^{-1}$, the restricted realization
sends such a diagram to $K_2/K_1$.
Its essential image is therefore $\Adm_\T(\Lambda)$.
By Proposition~\ref{prop:chz}, the image is independent of the
realization functor and equals $\C_\T$.
The remaining assertions follow from
Proposition~\ref{prop:quotient-description}.
\end{proof}

Since $\Ann_\Lambda(\Lambda/K)=K$ for every two-sided ideal $K$,
Proposition~\ref{prop:noohi-admissible} shows that $\Lambda/K$
is admissible precisely when $I_\Lambda(\T)\subseteq K$.
Thus $\Lambda/I_\Lambda(\T)$ is the largest admissible quotient
of the regular module: every admissible regular quotient is a
quotient of it.

\begin{exm}
	Let $\Gamma$ be a finite-dimensional algebra and let
	$(\T,\F)$ be a split torsion pair in $\mod\Gamma$. Every module $X\in\mod\Gamma$ decomposes
	as $X\simeq T\oplus F$ with $T\in\T$ and $F\in\F$.
	In Definition~\ref{def:tau-admissible}, take
	\[
	E=X,\quad K_1=0,\quad K_2=X,\quad L=F.
	\]
	Then
	$K_2/K_1\simeq X$, $K_2\cap L=F\in\F$, and
	$E/(K_1+L)\simeq T\in\T$.
	Thus every $\Gamma$-module is $\T$-admissible,
	and $\Adm_{\T}(\Gamma)=\mod\Gamma.$
	By Proposition~\ref{prop:noohi-admissible},
	\[
	\Adm_{\T}(\Gamma)
	=
	\mod\bigl(\Gamma/I_\Gamma(\T)\bigr)
	\]
	as the full subcategory of modules annihilated by
	$I_\Gamma(\T)$.  It follows that $I_\Gamma(\T)=0$.
\end{exm}

Applying Proposition~\ref{prop:noohi-admissible} to successive
quotients gives the following recursion.

\begin{theorem}\label{thm:hrs-ideal-recursion}
The HRS ideal filtration is determined recursively by
\begin{equation}\label{eq:hrs-ideal-recursion}
I_{n+1}/I_n=I_{\Lambda_n}(\T_n),\qquad
\Lambda_{n+1}=\Lambda_n/I_{\Lambda_n}(\T_n).
\end{equation}
All $I_n$ and $\C_n$ are independent of the chosen realization
functors.
\end{theorem}
\begin{proof}
Since $\C_{n+1}\subseteq\C_n=\mod\Lambda_n$ is the second-heart
image for $(\T_n,\F_n)$, Proposition~\ref{prop:noohi-admissible} gives
\[
I_{\Lambda_n}(\T_n)=\Ann_{\Lambda_n}\C_{n+1}
=\Ann_{\Lambda_n}\mod(\Lambda/I_{n+1})=I_{n+1}/I_n.
\]
The second equality in \eqref{eq:hrs-ideal-recursion} follows by
taking quotients.

Equivalently, $I_{n+1}$ is the inverse image of
$I_{\Lambda_n}(\T_n)$ under $\Lambda\twoheadrightarrow\Lambda_n$.
Since admissibility depends only on the torsion pair
(Definition~\ref{def:tau-admissible}) and the transported classes
are fixed by restriction of scalars, induction
from $I_0=0$ proves independence of all realization choices.
\end{proof}

\begin{remark}\label{rem:ext2-hrs-ideal}
For $F\in\F$ and $T\in\T$, put $F'=\Sigma_\B^{-1}(F[1])\in\A'$.
The identification
of first Yoneda extensions in a heart with degree-one morphisms gives
$\operatorname{YExt}^1_{\A'}(F',T)
\simeq\Hom_{\Db(\B)}(F[1],\Sigma_\B^2T)
\simeq\operatorname{YExt}^2_\B(F[1],T)$.
Combining this with the exact equivalence
$\A'\simeq\mod(\Lambda/I_1)$ and the fully faithful embedding of
\cite[Proposition~3.1]{CHZ} gives
\[
\operatorname{YExt}^2_\B(F[1],T)
\simeq\Ext^1_{\Lambda/I_1}(F,T)
\hookrightarrow\Ext^1_\Lambda(F,T).
\]
An extension $0\to T\to X\to F\to0$ lies in this image
exactly when $XI_1=0$.
Indeed, all three terms are then $\Lambda/I_1$-modules and both
maps are $\Lambda/I_1$-linear. Since any $\Lambda$-linear equivalence between
two such extensions is also $\Lambda/I_1$-linear, restriction of
scalars is injective on extension classes.
\end{remark}

\subsection{Height and the stopping criterion}\label{subsec:hrs-height-bounds}
The number of strict inclusions in the HRS ideal filtration bounds
the stage from which the realization functors are equivalences. We put
$\Lambda^{\mathrm e}=\Lambda^{\mathrm{op}}\otimes_k\Lambda.$
Two-sided ideals of $\Lambda$ are finite-length $\Lambda^{\mathrm e}$-modules.

\begin{definition}\label{def:three-invariants}
The \emph{HRS height} is
$h_\Lambda(\T)=\min\{n\ge0:I_{n+1}=I_n\}$.
For $h=h_\Lambda(\T)$, write $I_\infty=I_h$ for the stable
ideal.
\end{definition}

\begin{theorem}\label{cor:first-hrs-ideal-stopping}\label{thm:refined-height}
The following conditions are equivalent for an integer $n\geq 0$:
\begin{enumerate}
	\item[(1)]$I_{n+1}=I_n$;
		\item[(2)]$I_{\Lambda_n}(\T_n)=0$;
			\item[(3)]$G^{(2n+1)}$ is an equivalence.
\end{enumerate}
If $I_{n+1}=I_n$, then $I_n=I_m$ for any $m>n$. 
\end{theorem}
\begin{proof}
By Theorem~\ref{thm:hrs-ideal-recursion} and
Proposition~\ref{prop:quotient-description}, we have that
\[I_{\Lambda_n}(\T_n)=0
\quad\Longleftrightarrow\quad I_{n+1}=I_n
\quad\Longleftrightarrow\quad\C_{n+1}=\C_n.
\]
Thus $(1)\Longleftrightarrow (2)$.

By Proposition~\ref{prop:chz}, $\C_{n+1}=\C_n$ if and only if
$G^{(2n+1)}$ is an equivalence. Thus $(1)\Longleftrightarrow (3)$.
By Proposition~\ref{prop:absorption}, if $I_{n+1}=I_n$, then $I_n=I_m$ for any $m>n$.
\end{proof}

\begin{corollary}The HRS height of a torsion pair $(\T,\F)$ in $\mod\Lambda$ satisfies
	\[
	h_\Lambda(\T)\le
	\min\{d_\Lambda,\len_{\Lambda^{\mathrm e}}(J_\T)\}.
	\]
	Moreover,
	$\len_{\Lambda^{\mathrm e}}(J_\T)\le\len_\Lambda(J_\T)
	\le d_\Lambda$.
\end{corollary}
\begin{proof}

Theorem~\ref{thm:categorical-extension} gives
$h_\Lambda(\T)\le d_\Lambda$.
Before stabilization, we have a chain of bimodules
\[
0=I_0\subsetneq I_1\subsetneq\cdots\subsetneq I_h\subseteq J_\T
\]
whose first $h$ inclusions are strict. Hence
$h\le\len_{\Lambda^{\mathrm e}}(J_\T)$.
Since $J_\T^2=0$, we have $J_\T\subseteq\rad\Lambda$.
Every strict bimodule chain is also a strict right-module chain, so
$\len_{\Lambda^{\mathrm e}}(J_\T)\le\len_\Lambda(J_\T)
\le\len_\Lambda(\rad\Lambda)$.
Finally, the regular module decomposes as
$\Lambda_\Lambda\simeq\bigoplus_iP_i$.
Its radical therefore has length
$\sum_i(\len_\Lambda(P_i)-1)=d_\Lambda$.
\end{proof}
\begin{remark}\label{prop:hrs-filtration-tail}
Restarting the construction at $(\Lambda_s,\T_s)$ gives ideals
$I_{s+n}/I_s$.
This follows by induction from Theorem \ref{thm:hrs-ideal-recursion}.
Since the original filtration is strict before its height and
constant thereafter, the new height is
\[
h_{\Lambda_s}(\T_s)=\max\{h_\Lambda(\T)-s,0\}.
\]
In particular, the stable category is
$\mod(\Lambda/I_\infty)$ and its first HRS ideal is zero.
Its common annihilator is $J_\T/I_\infty$ by
Lemma \ref{lem:square-zero-annihilator}.
Note that the equality $I_{\Lambda/I_\infty}(\T_h)=0$ need not imply
$J_\T/I_\infty=0$.
\end{remark}

We record the behavior under products and for hereditary algebras.

\begin{proposition}\label{prop:hrs-height-product}
Let $\Lambda=\Lambda^{(a)}\times\Lambda^{(b)}$.
Every torsion pair decomposes uniquely into torsion pairs
$(\T^{(a)},\F^{(a)})$ and $(\T^{(b)},\F^{(b)})$ on the two factors.
If their HRS ideal filtrations are $(I_n^{(a)})$ and $(I_n^{(b)})$, then
\[
I_n=I_n^{(a)}\times I_n^{(b)},\qquad
h_\Lambda(\T)=
\max\{h_{\Lambda^{(a)}}(\T^{(a)}),
       h_{\Lambda^{(b)}}(\T^{(b)})\}.
\]
\end{proposition}
\begin{proof}
The central idempotents of the product decompose every module
into its two components. Since torsion and torsion-free classes
are closed under finite direct sums and summands, they decompose
uniquely as asserted.
Since submodules, intersections, sums, and quotients are computed
componentwise, Definition~\ref{def:tau-admissible} gives
\[
\Adm_\T(\Lambda)=
\Adm_{\T^{(a)}}(\Lambda^{(a)})\times
\Adm_{\T^{(b)}}(\Lambda^{(b)}).
\]
Taking annihilators proves the formula for $I_1$.
Induction using \eqref{eq:hrs-ideal-recursion} gives the formula
for $I_n$. The product filtration is constant precisely when both
factor filtrations are constant, which gives the height formula.
\end{proof}

\begin{proposition}\label{prop:hereditary-arbitrary-pair}
If $\Lambda$ is hereditary, then every torsion pair $(\T,\F)$ in
$\mod\Lambda$ satisfies $h_\Lambda(\T)\le1$.
\end{proposition}
\begin{proof}
Let $\B$ be the first HRS heart.
Since $\Lambda$ is hereditary, the usual identification of
extensions in a heart gives, for $T\in\T$ and $F\in\F$,
\[
 \Ext^1_\B(T,F[1])
 \simeq\Hom_{\Db(\Lambda)}(T,F[2])
 =\Ext^2_\Lambda(T,F)=0.
\]
Thus every torsion sequence for $(\F[1],\T)$ splits, and every
object of $\B$ is a direct sum of objects from $\F[1]$ and $\T$.
By \cite[Example~4.2(2)]{CHZ}, the realization for the forward tilt of
$\B$ at $(\F[1],\T)$ is an equivalence. Since $\A^{(2)}$ is this
heart shifted by $\Sigma_\B^{-1}$, $G^{(2)}$ is also an equivalence.
Under this equivalence, the forward tilt of $\A^{(2)}$ identifies
with the standard heart $\B$. Hence $G^{(3)}$ is an equivalence by
\cite[Corollary~2.8]{CHZ}, and Theorem~\ref{thm:refined-height}
gives $h_\Lambda(\T)\le1$.
\end{proof}
\section{Silting complexes and the exact stabilization index}
\label{sec:silting-hrs-ideals}

Let $\P=[P^{-1}\xrightarrow d P^0]$ be a two-term silting complex over
$\Lambda$, concentrated in degrees $-1,0$. The torsion pair induced by
$\P$ is
\[
 \T(\P)=\Fac H^0(\P),\qquad
 \F(\P)=\Sub H^{-1}(\nu\P).
\]
For $X\in\mod\Lambda$, one has $X\in\T(\P)$ if and only if
$\Hom(\P,X[1])=0$, and $X\in\F(\P)$ if and only if
$\Hom(\P,X)=0$.
This torsion pair is due to Hoshino--Kato--Miyachi \cite{HKM}; see also
\cite[Theorem~2.1 and Proposition~2.4]{BZ}. Put
$J(\P)=\Ann_\Lambda\T(\P)\cap\Ann_\Lambda\F(\P)$.
We identify the first HRS ideal with the kernel in the Buan--Zhou
silting theorem and compute the successive ideals from the
differential of $\P$.

\subsection{The first HRS ideal}
\label{subsec:bz-tower}
Put $\Gamma=\End(\P)$ and choose an approximation triangle
\begin{equation}\label{eq:bz-approximation-triangle}
 \Lambda\longrightarrow\P'\xrightarrow f\P''\longrightarrow\Lambda[1],
 \qquad \P',\P''\in\add\P.
\end{equation}
The first map is a left $\add\P$-approximation.
The Buan--Zhou construction gives a two-term silting complex
\[
 \Q=[\Hom(\P,\P')\xrightarrow{\Hom(\P,f)}\Hom(\P,\P'')]
\]
over $\Gamma$, together with an epimorphism
$\Phi_\P:\Lambda\twoheadrightarrow\Lambda'=\End(\Q)$
\cite[Theorem~1.1]{BZ}. Put $\Phi=\Phi_\P$ and write
$\Phi_*:\mod\Lambda'\hookrightarrow\mod\Lambda$ for restriction
of scalars.  The kernel
$\mathfrak b_\Lambda(\P)$ of $\Phi_\P$
is generated by composites
\begin{equation}\label{eq: gen-silting}
 \Lambda\longrightarrow\P_I\longrightarrow\P_{II}[-1]
 \longrightarrow\Lambda,\qquad \P_I,\P_{II}\in\add\P,
\end{equation}
viewed as elements of $\End_\Lambda(\Lambda_\Lambda)=\Lambda$
\cite[Proposition~4.1]{BZ}. Hence this ideal depends only on
$\add\P$.

\begin{lemma}\label{cor:first-hrs-ideal-silting}
The first HRS ideal satisfies
$I_\Lambda(\T(\P))=\mathfrak b_\Lambda(\P)$.
Moreover, $\mathfrak b_\Lambda(\P)=0$ if and only if $\P$ is tilting.
\end{lemma}
\begin{proof}
Let $\C(\P)$ (resp. $\C(\Q)$) denote the forward HRS heart associated to $\P$ (resp. $\Q$). The equivalence
\[
 \Hom(\P,-):\C(\P)\xrightarrow{\sim}\mod\Gamma
\]
takes the torsion pair $(\F(\P)[1],\T(\P))$ on $\C(\P)$
to $(\T(\Q),\F(\Q))$
by \cite[Theorem~2.3 and Proposition~3.8]{BZ}.
Thus the backward second heart identifies with
$\mathcal H=\C(\Q)[-1]\subseteq\Db(\Gamma)$, and the heart
equivalence for $\Q$ gives
\[
 \Psi:\mathcal H\xrightarrow{\sim}\mod\Lambda',\quad
\Psi(X)=\Hom_{\mathrm D(\Gamma)}(\Q,X[1]).
\]

To compute the image of $\mathcal H$ in $\mod\Lambda$, set
\[
 \begin{gathered}
 E=\tau_{\le0}\End_{\mathrm{dg},\Lambda}(\P),\qquad
 p:E\longrightarrow H^0(E)=\Gamma,\\
 F=\mathbf R\Hom_\Lambda(\P,-):\mathrm D(\Lambda)
 \xrightarrow{\sim}\mathrm D(E).
 \end{gathered}
\]
Here $\mathrm D(-)$ denotes the unbounded derived category.
The truncation $\tau_{\le0}$ denotes the good dg truncation.
Since $\P$ is silting,
$H^{>0}\End_{\mathrm{dg},\Lambda}(\P)=0$. Hence the natural inclusion
$E\hookrightarrow\End_{\mathrm{dg},\Lambda}(\P)$ is a quasi-isomorphism.
Since $\P$ is a compact generator, derived Morita theory gives
the equivalence $F$ \cite[Section~6.1]{Keller}.
Write $p^*=-\otimes_E^{\mathbf L}\Gamma$ for extension of scalars
and $p_*$ for restriction of scalars.

Put $G_0=F^{-1}p_*\colon\mathrm D(\Gamma)\to\mathrm D(\Lambda)$.
Fix $X\in\Db(\Gamma)$. The full subcategory of objects
$Y\in\mathrm D(E)$ for which
$\Hom_{\mathrm D(E)}(Y,p_*X[j])$ is finite-dimensional for every $j$
and vanishes for all but finitely many $j$ is thick.
It contains $E$, since
$\Hom_{\mathrm D(E)}(E,p_*X[j])\simeq H^j(X)$.
Since $F(\Lambda)\in\thick(E)$, the groups
\[
 H^j(G_0X)\simeq\Hom_{\mathrm D(E)}(F(\Lambda),p_*X[j])
\]
are also finite-dimensional and vanish for all but finitely
many $j$. Hence $G_0$ restricts to
$\Db(\Gamma)\to\Db(\Lambda)$.

On objects with finite-dimensional total cohomology, the standard
heart of $\mathrm D(E)$ is $\mod \Gamma$.
By \cite[Lemma~5.3(a)]{KY}, $F$ restricts to an equivalence
$\C(\P)\xrightarrow{\sim}\mod\Gamma$. Since $p_*$ leaves underlying
complexes unchanged, it is $t$-exact and induces the identity
on this standard heart. Thus $G_0$ is $t$-exact
for the corresponding $t$-structures and induces the inverse of
$\Hom(\P,-):\C(\P)\to\mod\Gamma$. With this identification
of hearts, $G_0$ is a realization functor; see also
\cite[Lemmas~2.13--2.14, Proposition~2.15 and Section~3]{XYZ}.

By \cite[Theorem~4.5 and Remark~4.8]{XYZ},
$\Q\simeq p^*F(\Lambda)[1]$.
Moreover, the induced map
\[
 \Lambda=\End_{\mathrm D(E)}(F(\Lambda)[1])
 \longrightarrow\End_{\mathrm D(\Gamma)}(\Q)=\Lambda'
\]
is precisely $\Phi_\P$.
For $X\in\mathcal H$, Proposition~\ref{prop:chz} gives
$G_0X\in\mod\Lambda$. Adjunction now gives
\[
 \Phi_*\Psi(X)
 \simeq\Hom_{\mathrm D(E)}(F(\Lambda)[1],p_*X[1])
 \simeq\Hom_{\mathrm D(\Lambda)}(\Lambda,G_0X)
 \simeq G_0X.
\]
These are isomorphisms of right $\Lambda$-modules. Indeed, for $a\in\Lambda$,
naturality of adjunction identifies precomposition by $\Phi_\P(a)$
with precomposition by left multiplication $l_a$ on $\Lambda$;
for $m:\Lambda\to G_0X$, evaluation at $1$ gives
$(m\circ l_a)(1)=m(a)=m(1)a$.

Since $\Psi$ is essentially surjective, $G_0(\mathcal H)=
\Phi_*(\mod\Lambda')$. By Proposition~\ref{prop:chz}, this is
the intrinsic second-heart image, whose annihilator is
$I_\Lambda(\T(\P))$. Therefore
\[
 I_\Lambda(\T(\P))
 =\Ann_\Lambda\Phi_*(\mod\Lambda')
 =\mathfrak b_\Lambda(\P).
\]
The last assertion follows from \cite[Corollary~4.3]{BZ}.
\end{proof}

\begin{corollary}\label{cor:silting-faithfulness-criterion}
If $H^0(\P)\oplus H^{-1}(\nu\P)$ is faithful over $\Lambda$,
then $\P$ is tilting and $\mathfrak b_\Lambda(\P)=0$.
\end{corollary}
\begin{proof}
Since a module has the same annihilator as the class of its factor
modules and as the class of its submodules,
\[
 J(\P)=\Ann_\Lambda\bigl(H^0(\P)\oplus H^{-1}(\nu\P)\bigr)=0.
\]
Since Lemma~\ref{lem:square-zero-annihilator} gives
$I_1\subseteq J(\P)=0$, Lemma~\ref{cor:first-hrs-ideal-silting} gives
$\mathfrak b_\Lambda(\P)=0$ and shows that $\P$ is tilting.
\end{proof}

Iterate the Buan--Zhou construction, starting with
$(\Lambda_0^{\mathrm{BZ}},\P_0)=(\Lambda,\P)$:
\[
 (\Lambda_i^{\mathrm{BZ}},\P_i),\quad(\Gamma_{i},\Q_i),\quad
 (\Lambda_{i+1}^{\mathrm{BZ}},\P_{i+1}).
\]
Here $\Gamma_i=\End(\P_i)$ and
$\Lambda_{i+1}^{\mathrm{BZ}}=\End(\Q_i)$.
There are quotient maps $\Lambda_i^{\mathrm{BZ}}\twoheadrightarrow
\Lambda_{i+1}^{\mathrm{BZ}}$ and $\Gamma_{i}\twoheadrightarrow \Gamma_{i+1}$.
One increase in $i$ consists of two successive HRS tilts:
from $\mod\Lambda_i^{\mathrm{BZ}}$ to $\mod \Gamma_{i}$, and then to
$\mod\Lambda_{i+1}^{\mathrm{BZ}}$.

\begin{corollary}\label{lem:bz-hrs-tower}
\label{cor:bz-ideal-increment}
The silting tower identifies the alternating hearts and their
transported torsion pairs as
\[
 \A^{(2i)}\simeq\mod\Lambda_i,\qquad
 \A^{(2i+1)}\simeq\mod \Gamma_{i},
\]
where $\Lambda_i=\Lambda/I_i$ and the pairs are induced by
$\P_i$ and $\Q_i$, respectively.
Under the composite quotient maps,
\begin{equation}\label{eq:bz-ideal-increment}
 \Lambda_i^{\mathrm{BZ}}=\Lambda_i,\qquad
 I_{i+1}/I_i=\mathfrak b_{\Lambda_i}(\P_i).
\end{equation}
Thus $I_{i+1}=I_i$ if and only if $\P_i$ is tilting.
\end{corollary}
\begin{proof}
The first-heart equivalence gives $\A^{(1)}\simeq\mod \Gamma_{0}$.
By Lemma~\ref{cor:first-hrs-ideal-silting}, the second-heart image is
$\mod(\Lambda/\mathfrak b_\Lambda(\P))$.
The transported pair on this quotient is induced by $\P_1$
by \cite[Proposition~3.8]{BZ}.
Theorem~\ref{thm:hrs-ideal-recursion} identifies
$\Lambda/\mathfrak b_\Lambda(\P)$ with $\Lambda/I_1$
as quotients of $\Lambda$.
Repeating this argument over $\Lambda_i$ gives
\eqref{eq:bz-ideal-increment}. The last assertion follows from
Lemma~\ref{cor:first-hrs-ideal-silting}.
\end{proof}
By Corollary~\ref{cor:bz-ideal-increment}, the ideal $I_{i+1}$ is the inverse image of the first HRS ideal
$\mathfrak b_{\Lambda_i}(\P_i)$ under
$\Lambda\twoheadrightarrow\Lambda_i$.

\subsection{Reduction modulo an ideal}
For an ideal $J$ of $\Lambda$, the quotient map
$\Lambda\twoheadrightarrow\Lambda/J$ sends a right module $X$
to $X\otimes_\Lambda\Lambda/J\simeq X/XJ$. Applied termwise to
$\P$, it gives
\[
 \P\otimes_\Lambda\Lambda/J
 =[P^{-1}/P^{-1}J\xrightarrow{\bar d} P^0/ P^0J],\qquad
 \bar d(u+P^{-1}J)=d(u)+ P^0J.
\]
Since $d(P^{-1}J)\subseteq  P^0J$, the map $\bar d$ is well-defined.
In matrix form, its entries are those of $d$ modulo $J$.
Since the terms of $\P$ are projective, this also computes the
derived tensor product. The following theorem allows us to use
these quotients in place of successive approximation triangles.

\begin{theorem}\label{thm:hrs-base-change}
For every ideal $J$ of $\Lambda$, the complex
$\P\otimes_\Lambda\Lambda/J$ is two-term silting.
Put $\Lambda_1=\Lambda/\mathfrak b_\Lambda(\P)$, and let
$\P^+$ be the Buan--Zhou successor of $\Q$. Then
for $\overline\P=\P\otimes_\Lambda\Lambda_1$ one has
\begin{equation}\label{eq:hrs-base-change-add}
 \add\overline\P=\add\P^+
 \quad\text{in }\Kb(\proj\Lambda_1).
\end{equation}
\end{theorem}
\begin{proof}
Write $C=\Lambda/J$ and
$\P_C=[P^{-1}/P^{-1}J\xrightarrow{d_C} P^0/ P^0J]$.
By \cite[Corollary~2.4]{Breaz}, $\P_C$ is silting.

Now take $J=\mathfrak b_\Lambda(\P)$. Regard $X\in\mod\Lambda_1$ also as a
$\Lambda$-module via the quotient map. Since $\P$ is a bounded
projective complex, tensor--restriction adjunction gives, for $j=0,1$,
\[
 \Hom_{\Db(\Lambda_1)}(\overline\P,X[j])
 \simeq\Hom_{\Db(\Lambda)}(\P,X[j]).
\]
Thus $X$ belongs to $\T(\overline\P)$, respectively
$\F(\overline\P)$, exactly when its restriction belongs to
$\T(\P)$, respectively $\F(\P)$.
Set $\mathcal X(\Q)=\Hom(\Q,\F(\Q)[1])$ and
$\mathcal Y(\Q)=\Hom(\Q,\T(\Q))$.
By \cite[Theorem~4.4]{BZ}, restriction identifies these two classes
with $\T(\P)$ and $\F(\P)$.
By \cite[Proposition~3.8]{BZ}, $(\mathcal X(\Q),\mathcal Y(\Q))$
is also the pair induced by $\P^+$. Consequently
\begin{equation}\label{eq:base-change-torsion-class}
 (\T(\overline\P),\F(\overline\P))
 =(\T(\P^+),\F(\P^+)).
\end{equation}
The bijection between functorially finite torsion classes and
basic two-term silting complexes \cite[Theorems~2.7 and~3.2]{AIR}
now identifies their basic representatives. Therefore their
additive closures agree, as asserted.
\end{proof}

\begin{corollary}\label{cor:iterated-hrs-base-change}
At every even stage of HRS tilting,
\begin{equation}\label{eq:direct-hrs-ideal-recursion}
 \add\P_i=\add(\P\otimes_\Lambda\Lambda_i),\qquad
 I_{i+1}/I_i=
 \mathfrak b_{\Lambda_i}(\P\otimes_\Lambda\Lambda_i).
\end{equation}
\end{corollary}
\begin{proof}
The assertion holds for $i=0$.
Assume it holds at $i$. Theorem~\ref{thm:hrs-base-change}
identifies $\add\P_{i+1}$ with
$\add(\P_i\otimes_{\Lambda_i}\Lambda_{i+1})$.
Since tensor product preserves split summands, we replace $\P_i$ by
$\P\otimes_\Lambda\Lambda_i$ in this equality.
Associativity then gives
$\add\P_{i+1}=\add(\P\otimes_\Lambda\Lambda_{i+1})$.
Since the ideal \eqref{eq: gen-silting} depends only on the additive
closure, \eqref{eq:bz-ideal-increment} gives the formula for each increment.
\end{proof}

Equality of additive closures identifies the torsion pairs, first
HRS ideals and tilting property of $\P_i$ and
$\P\otimes_\Lambda\Lambda_i$. It need not identify their summand
multiplicities or endomorphism algebras. We use the actual Buan--Zhou complexes for the quotient maps
in the tower, and use the reduced complexes to compute the ideals.

\subsection{Two parity heights}
\label{subsec:parity-stabilization}
The odd hearts give a second ideal filtration. Set
$K_i=\ker(\Gamma_{0}\twoheadrightarrow \Gamma_{i})$, so that
\begin{equation}\label{eq:odd-hrs-ideal-recursion}
 \Gamma_{i}=\Gamma_{0}/K_i,\qquad
 K_{i+1}/K_i=\mathfrak b_{\Gamma_{i}}(\Q_i).
\end{equation}
Write $h_{\mathrm{ev}}=h_\Lambda(\T(\P))$ and
$h_{\mathrm{odd}}=h_{\Gamma_{0}}(\T(\Q_0))$.
Thus these are the first indices at which $\P_i$ and $\Q_i$,
respectively, become tilting.

\begin{proposition}\label{prop:exact-parity-stabilization}
The heights satisfy $h_{\mathrm{odd}}\le h_{\mathrm{ev}}
\le h_{\mathrm{odd}}+1$. The least $s\ge1$ such that all
$G^{(m)}$, $m\ge s$, are equivalences is
\begin{equation}\label{eq:exact-parity-stabilization}
 \sigma_\Lambda(\T(\P))
 =\max\{2h_{\mathrm{ev}},2h_{\mathrm{odd}}+1\}
 =h_{\mathrm{ev}}+h_{\mathrm{odd}}+1.
\end{equation}
\end{proposition}
\begin{proof}
By \cite[Corollary~4.3]{BZ}, the Buan--Zhou successor of a tilting
complex is tilting. Hence
\[
 \P_i\text{ tilting}\ \Longrightarrow\ \Q_i\text{ tilting}
 \ \Longrightarrow\ \P_{i+1}\text{ tilting}.
\]
Taking $i=h_{\mathrm{ev}}$ in the first implication and
$i=h_{\mathrm{odd}}$ in the second gives
$h_{\mathrm{odd}}\le h_{\mathrm{ev}}\le h_{\mathrm{odd}}+1$.

The realization $G^{(2i+1)}$ is the forward realization for
$\P_i$. The realization $G^{(2i+2)}$ is backward; shifting its
source heart gives the forward heart for $\Q_i$. Since a shift
is an equivalence, Theorem~\ref{cor:first-hrs-ideal-stopping} and
Lemma~\ref{cor:first-hrs-ideal-silting} give
\[
 \begin{aligned}
 G^{(2i+1)}\text{ is an equivalence}
 &\ \Longleftrightarrow\ \P_i\text{ is tilting}
 \ \Longleftrightarrow\ i\ge h_{\mathrm{ev}},\\
 G^{(2i+2)}\text{ is an equivalence}
 &\ \Longleftrightarrow\ \Q_i\text{ is tilting}
 \ \Longleftrightarrow\ i\ge h_{\mathrm{odd}}.
 \end{aligned}
\]
If $h_{\mathrm{ev}}>0$, the last failing odd index is
$2h_{\mathrm{ev}}-1$.
If $h_{\mathrm{odd}}>0$, the last failing even index is
$2h_{\mathrm{odd}}$.
Taking one more than the last failure gives the maximum in
\eqref{eq:exact-parity-stabilization}.
The same formula gives $1$ when both heights vanish.
Finally,
$h_{\mathrm{ev}}=h_{\mathrm{odd}}$ or
$h_{\mathrm{ev}}=h_{\mathrm{odd}}+1$.
In each case the maximum equals the stated sum.
\end{proof}

\subsection{A matrix formula}
\label{subsec:hrs-ideal-algorithm}
Choose decompositions $ P^0=\bigoplus_{i=1}^r e_i\Lambda$ and
$P^{-1}=\bigoplus_{j=1}^s f_j\Lambda$, where $e_i,f_j$ are primitive
idempotents and repetitions are allowed.
Write $\mathbf d$ for the matrix of $d$ with respect to
these decompositions. A map $L: P^0\to P^{-1}$ has entries
$a_{ji}\in f_j\Lambda e_i$. 
We define a $k$-linear map
\begin{equation}\label{eq:matrix-defect-kernel}
	\Theta_{\mathbf{d}}:\Hom_\Lambda(P^0,P^{-1})
	\longrightarrow\End_\Lambda(P^{-1})\oplus\End_\Lambda(P^0),\quad L\mapsto(L\mathbf{d},\mathbf{d}L).
\end{equation}
A chain map $\P\to\P[-1]$ has just one possibly nonzero component
$L:P^0\to P^{-1}$, and its chain-map conditions are $Ld=dL=0$.
There are no nonzero homotopies between such maps. Hence
\begin{equation}\label{eq:negative-self-hom-kernel}
	\Hom_{\Kb(\proj\Lambda)}(\P,\P[-1])
	\simeq \ker\Theta_{\mathbf{d}}.
\end{equation}
\begin{lemma}\label{lem:matrix-hrs-ideal}The first HRS ideal is given by
\begin{equation}\label{eq:matrix-hrs-ideal-equality}
 I_\Lambda(\T(\P))=\mathfrak b_\Lambda(\P)
 =\langle a_{ji}\mid L=(a_{ji})\in\ker\Theta_{\mathbf d}\rangle,
\end{equation}
where the brackets denote the two-sided ideal generated by the entries.
\end{lemma}
\begin{proof}
Let $\mathfrak a$ be the ideal generated by the entries of all
$L\in\ker\Theta_{\mathbf d}$. Given module maps
$u:\Lambda\to P^0$ and
$v:P^{-1}\to\Lambda$, the composite
\[
 \Lambda\xrightarrow{u}\P\xrightarrow{L}\P[-1]
 \xrightarrow{v}\Lambda
\]
is defined by chain maps. By \eqref{eq: gen-silting},
$vLu(1)\in\mathfrak b_\Lambda(\P)$. Taking $u(1)=e_i$ in the
$i$th summand of $ P^0$, and taking $v$ to be the projection onto
$f_j\Lambda$ followed by its inclusion into $\Lambda$, gives
$vLu(1)=a_{ji}$. Hence $\mathfrak a\subseteq\mathfrak b_\Lambda(\P)$.

Conversely, choose split inclusions of $\P_I$ and $\P_{II}$ in
\eqref{eq: gen-silting} into finite direct sums of $\P$.
Inserting the corresponding projections expresses each generator
of $\mathfrak b_\Lambda(\P)$ as a sum of terms $vLu(1)$ as above.
Write $u(1)=(x_i)_i$ and $v(f_j)=y_j$, where
$x_i\in e_i\Lambda$ and $y_j\in\Lambda f_j$. Then
\[
 vLu(1)=\sum_{i,j}y_ja_{ji}x_i\in\mathfrak a.
\]
Thus $\mathfrak b_\Lambda(\P)\subseteq\mathfrak a$. Together with
Lemma~\ref{cor:first-hrs-ideal-silting}, this proves
\eqref{eq:matrix-hrs-ideal-equality}.
\end{proof}

For a bound quiver algebra, expand the entries of $L$ in bases
of the spaces $f_j\Lambda e_i$. The equations
$L\mathbf d=\mathbf dL=0$ are
linear over $k$. The entries of a basis of their solution space
generate the first HRS ideal. Repeating this calculation with
$\mathbf d$
reduced modulo $I_i$ gives $I_{i+1}/I_i$ by
Corollary~\ref{cor:iterated-hrs-base-change}. The calculation stops
when this ideal is zero. 
\section{Kernel filtrations and the universal tilting quotient}
\label{sec:kernel-lifting}

We replace the repeated matrix calculations of the previous section
by a filtration of $H^{-1}(\P)$, and then show that HRS height is not a derived invariant.

\subsection{General kernel filtration}
\label{subsec:kernel-filtration}

Let $\P=[P^{-1}\xrightarrow d P^0]$ be two-term
silting. Put $\T=\Fac H^0(\P)$.
Write $t(X)$ for the torsion radical of a module $X$, and put
$\F=\{X:\Hom_\Lambda(H^0(\P),X)=0\}$.
Since every object of $\T$ is a quotient of a finite sum of copies
of $H^0(\P)$, one has
\[
t(X)=\sum_{f:H^0(\P)\to X}f(H^0(\P)).
\]
Let $\{I_n\}$ be the HRS ideal filtration, with $I_0=0$, and set
\[
\Lambda_n=\Lambda/I_n,\qquad P_n^{-1}=P^{-1}/P^{-1}I_n,\qquad  P_n^0= P^0/ P^0I_n.
\]
By Corollary~\ref{cor:iterated-hrs-base-change}, we may use the
add-equivalent representatives
\[
\P_n=\P\otimes_\Lambda\Lambda_n=[P_n^{-1}\xrightarrow{d_n} P_n^0],
\qquad d_n(u+P^{-1}I_n)=d(u)+ P^0I_n.
\]
We abbreviate $h_\Lambda(\T(\P))$ to $h_\Lambda(\P)$.
\begin{lemma}\label{lem:hom-id}
	For every two-term complex $\P$ of finitely generated projectives,
	there is a natural isomorphism
	\begin{equation}\label{eq:kernel-negative-hom}
		\Hom_{\Kb(\proj\Lambda)}(\P,\P[-1])
		\simeq \Hom_\Lambda(H^0(\P),H^{-1}(\P)).
	\end{equation}
\end{lemma}
\begin{proof}
	This follows from \cite[Lemmas~2.3(2) and~2.5]{HKM}.
	Explicitly, let $\pi: P^0\twoheadrightarrow H^0(\P)$ and
	$\iota:H^{-1}(\P)\hookrightarrow P^{-1}$ be the canonical maps.
	A module map $f:H^0(\P)\to H^{-1}(\P)$ corresponds to the negative
	chain map $L=\iota f\pi$. Conversely, every
	$L: P^0\to P^{-1}$ with $Ld=dL=0$ factors uniquely in this way.
	There are no nonzero homotopies in degree $-1$.
\end{proof}

The equality $d(P^{-1}I_n)= P^0I_n$ will identify
$H^{-1}(\P_n)$ with a quotient of $H^{-1}(\P)$.
We begin with the first HRS ideal.

\begin{lemma}\label{lem:kernel-first-identities}
	For $I=I_1$ one has
	\[
	I\in\T,\quad  d(P^{-1}I)= P^0I\text{ and }
	t(H^{-1}(\P))\subseteq H^{-1}(\P)\cap P^{-1}I.
	\]
\end{lemma}
\begin{proof}
	Let $\pi: P^0\twoheadrightarrow H^{0}(\P)$ and
	$\iota:H^{-1}(\P)\hookrightarrow P^{-1}$ be the canonical maps.
	Lemmas~\ref{lem:hom-id} and~\ref{lem:matrix-hrs-ideal} give
	\begin{equation}\label{eq:kernel-first-coefficients}
	I=\sum_{\phi:P^{-1}\to\Lambda}\sum_{f:H^{0}(\P)\to H^{-1}(\P)}
		\phi\iota f(H^{0}(\P)).
	\end{equation}
	Thus $I\in\Fac H^{0}(\P)=\T$.
	Lemma~\ref{lem:matrix-hrs-ideal} shows that every negative chain
	map $L=\iota f\pi$ has image in $P^{-1}I$.
	Taking the sum of these images gives
	$t(H^{-1}(\P))\subseteq H^{-1}(\P)\cap P^{-1}I$.
	
	It remains to prove $ P^0I\subseteq d(P^{-1}I)$.
	A generating term of $ P^0I$ has the form $v\phi L(v')$,
	where $v,v'\in P^0$ and $\phi:P^{-1}\to\Lambda$.
	Set $g(u)=v\phi(u)$.
	Since $\P$ is presilting, there exist
	$X\in\End_\Lambda( P^0)$ and $Y\in\End_\Lambda(P^{-1})$
	with $g=Xd+dY$. Hence
	\[
	v\phi L(v')=gL(v')=dYL(v')\in d(P^{-1}I),
	\]
	since $dL=0$ and $L( P^0)\subseteq P^{-1}I$.
	The reverse inclusion follows from $\Lambda$-linearity.
\end{proof}

The submodules $H^{-1}(\P)\cap P^{-1}I_n$ give the following description
of the cohomology of $\P_n$.

\begin{theorem}\label{thm:main-kernel-filtration}
	Keep the notation above and set $R_n=H^{-1}(\P)\cap P^{-1}I_n$.
	We have:
	\begin{enumerate}[label=(\arabic*),leftmargin=2.2em]
		\item $H^0(\P_n)\simeq H^{0}(\P)$;
		\item $H^{-1}(\P_n)\simeq H^{-1}(\P)/R_n$;
		\item $\Hom_{\Kb(\proj\Lambda_n)}(\P_n,\P_n[-1])
		\simeq\Hom_\Lambda(H^{0}(\P),H^{-1}(\P)/R_n)$;
		\item The two-term silting complex $\P_n$ is tilting if and only if
		$H^{-1}(\P)/R_n\in\F$.
	\end{enumerate}
	In particular,
	\begin{equation}\label{eq:kernel-height}
		h_\Lambda(\P)=\min\{n\ge0:\Hom_\Lambda(H^{0}(\P),H^{-1}(\P)/R_n)=0\}
		=\min\{n\ge0:H^{-1}(\P)/R_n\in\F\}.
	\end{equation}
\end{theorem}
\begin{proof}
	We first prove $d(P^{-1}I_n)= P^0I_n$ by induction on $n$.
	The case $n=0$ is clear.
	By Corollary~\ref{cor:iterated-hrs-base-change}, $I_{n+1}/I_n$ is the
	first HRS ideal of $\P_n$. Applying Lemma~\ref{lem:kernel-first-identities}
	over $\Lambda_n$ and then using the induction hypothesis gives
	\[
	 P^0I_{n+1}
	=d(P^{-1}I_{n+1})+ P^0I_n
	=d(P^{-1}I_{n+1}).
	\]
	
	The termwise short exact sequence of complexes
	\[
	0\longrightarrow\P I_n\longrightarrow\P
	\longrightarrow\P_n\longrightarrow0,
	\qquad \P I_n=[P^{-1}I_n\xrightarrow d P^0I_n],
	\]
	has $H^0(\P I_n)=0$ and $H^{-1}(\P I_n)=R_n$.
	Its cohomology sequence therefore gives the natural identifications
	in (1) and (2).
	Applying Lemma~\ref{lem:hom-id} to $\P_n$ proves (3), since
	$\Lambda_n$-linear and $\Lambda$-linear maps between these quotient
	modules coincide.  (4) follows from (3).
\end{proof}

\begin{remark}\label{rem:kernel-ext-obstruction}
	For $n\ge1$, every map $H^{0}(\P)\to H^{-1}(\P)$ has image in
	$t(H^{-1}(\P))\subseteq R_1\subseteq R_n$ by
	Lemma~\ref{lem:kernel-first-identities}. The connecting map for
	$0\to R_n\to H^{-1}(\P)\to H^{-1}(\P)/R_n\to0$ therefore yields an isomorphism
	\begin{equation}\label{eq:kernel-ext-description}
		\Hom_\Lambda(H^{0}(\P),H^{-1}(\P)/R_n)\xrightarrow[\delta_n]{\sim}
		\ker\bigl[\Ext^1_\Lambda(H^{0}(\P),R_n)\longrightarrow
		\Ext^1_\Lambda(H^{0}(\P),H^{-1}(\P))\bigr].
	\end{equation}
	Since $\delta_n$ is injective, no nonzero map on the left
	lifts to $H^{-1}(\P)$.
	By Theorem~\ref{thm:main-kernel-filtration}, these are precisely the
	new negative maps of $\P_n$; their matrix coefficients generate
	$I_{n+1}/I_n$ by Lemma~\ref{lem:matrix-hrs-ideal}.
\end{remark}

The filtration $\{R_n\}$ also bounds the number of nonzero filtration
layers and characterizes the terminal quotient.

\begin{proposition}\label{prop:main-universal-tilting}
	\label{cor:main-two-sided-kernel-bound}
	If $J$ is an ideal of $\Lambda$ and $\P\otimes_\Lambda\Lambda/J$ is tilting,
	then $I_\infty\subseteq J$ and $\Lambda/I_\infty$ is the largest quotient over which
	reduction makes $\P$ tilting.
	Moreover,
	\[
	h_\Lambda(\P)\le
	\min\{\len_\Lambda H^{-1}(\P),
	\len_{\Lambda^{\mathrm{op}}}\Hom_\Lambda(H^{0}(\P),\Lambda)\}.
	\]
\end{proposition}
\begin{proof}
	Put $h=h_\Lambda(\P)$.
	Theorem~\ref{thm:main-kernel-filtration} shows that $\P_h$ is tilting.
	Suppose $\P_J=\P\otimes_\Lambda\Lambda/J$ is tilting.
	We prove $I_n\subseteq J$ by induction. The case $n=0$ is clear.
	Since $I_n\subseteq J$, the quotient map to $\Lambda/J$ factors through
	$\Lambda_n$, and
	\[
	\P_n\otimes_{\Lambda_n}\Lambda/J
	\simeq\P\otimes_\Lambda\Lambda/J=\P_J.
	\]
	The ideal $I_{n+1}/I_n=\mathfrak b_{\Lambda_n}(\P_n)$ is generated by
	the composites in \eqref{eq: gen-silting}, taken over $\Lambda_n$.
	After tensoring with $\Lambda/J$, their middle terms belong to
	$\add\P_J$ and $\add\P_J[-1]$, respectively.
	Since $\P_J$ is tilting, the morphisms between these terms are zero.
	Thus every generating composite maps to zero under
	$\Lambda_n\twoheadrightarrow\Lambda/J$, whose kernel is $J/I_n$.
	It follows that $I_{n+1}/I_n\subseteq J/I_n$, and hence $I_{n+1}\subseteq J$.
	Induction gives $I_\infty\subseteq J$, proving the universal property.
	
	For $n<h$, apply Lemma~\ref{lem:kernel-first-identities} over
	$\Lambda_n$. Under the natural identifications of
	Theorem~\ref{thm:main-kernel-filtration}, it gives
	\[
	0\ne t(H^{-1}(\P)/R_n)\subseteq R_{n+1}/R_n.
	\]
	Thus $0=R_0\subsetneq R_1\subsetneq\cdots\subsetneq R_h\subseteq H^{-1}(\P)$,
	and $h\le\len_\Lambda H^{-1}(\P)$.
	
	Lemma~\ref{lem:kernel-first-identities}  gives $I_{n+1}/I_n\in\T_n=\T$.
	By extension closure, $I_n\in\T=\Fac H^{0}(\P)$ for all $n$.
	Since $I_{n+1}$ is generated by images of maps from $H^{0}(\P)$,
	$I_n\subsetneq I_{n+1}$ implies
	\[
	\Hom_\Lambda(H^{0}(\P),I_n)\subsetneq\Hom_\Lambda(H^{0}(\P),I_{n+1}).
	\]
	Indeed, equality would force all those images into $I_n$.
	For $0\le n<h$, these form a strict chain of left
	$\Lambda$-submodules of $\Hom_\Lambda(H^{0}(\P),\Lambda)$, proving
	$h\le\len_{\Lambda^{\mathrm{op}}}\Hom_\Lambda(H^{0}(\P),\Lambda)$.
\end{proof}

The following condition on  the second syzygy $\Omega^2_\Lambda H^{0}(\P)$ implies
$H^{-1}(\P)/R_1\in\F$.

\begin{proposition}\label{thm:main-syzygy-lifting}
	Let $N=\Omega^2_\Lambda H^{0}(\P)$ in a minimal projective resolution of $H^{0}(\P)$.
	If every composition factor of $N/t(N)$ belongs to $\F$, then
	$h_\Lambda(\P)\le1$.
\end{proposition}
\begin{proof}
	Replace $\P$ by a minimal representative and write
	\[
	\P=[P^{-1}_{\mathrm{pr}}\xrightarrow{d_0} P^0]\oplus W[1],
	\]
	where the first summand is a minimal projective presentation of $H^{0}(\P)$
	and $W$ is projective; see \cite[proof of Proposition~3.6]{AIR}.
	Thus $H^{-1}(\P)=N\oplus W$.
	Put $I=I_1$ and $S=N\cap P^{-1}_{\mathrm{pr}}I$.
	Then
	\[
	R_1=S\oplus WI,
	\qquad H^{-1}(\P)/R_1\simeq(N/S)\oplus(W/WI).
	\]
	Lemma~\ref{lem:kernel-first-identities} gives
	$t(N)\subseteq S$ and $t(W)\subseteq WI$.
	Since $W$ is projective, $WI\in\add I\subseteq\T$; hence
	$WI\subseteq t(W)$ and therefore $WI=t(W)$.
	It follows that $W/WI\in\F$.
	
	The module $N/S$ is a quotient of $N/t(N)$, so all its composition
	factors belong to $\F$. Extension closure of $\F$
	gives $N/S\in\F$.
	Consequently $H^{-1}(\P)/R_1\in\F$, and
	Theorem~\ref{thm:main-kernel-filtration} yields $h_\Lambda(\P)\le1$.
\end{proof}
\begin{corollary}\label{cor:collapse}
With $N$ as in Proposition~\ref{thm:main-syzygy-lifting}, if $N$ is semisimple, then
$h_\Lambda(\P)\le1$. In particular, when
$\rad^2\Lambda=0$ or $\operatorname{pd}_\Lambda H^{0}(\P)\le1$, $N$ is semisimple.
\end{corollary}
\begin{proof}
	If $N$ is semisimple, so is $N/t(N)\in\F$, and its composition
	factors belong to $\F$ as direct summands.
	If $\rad^2\Lambda=0$, then
	$N\subseteq\rad P^{-1}_{\mathrm{pr}}$ is semisimple.
	If $\operatorname{pd}_\Lambda  H^{0}(\P)\le1$, then $N=0$.
\end{proof}
\subsection{An example}
Put $h_{\mathrm{HRS}}(\Lambda)=\sup_{\P}h_\Lambda(\P)$, where
$\P$ ranges over basic two-term silting complexes over $\Lambda$.
Proposition~\ref{prop:hereditary-arbitrary-pair} gives
$h_{\mathrm{HRS}}(\Lambda)\le1$ for every hereditary algebra $\Lambda$.
The following example shows that this bound is not preserved by
derived equivalence.
\begin{exm}
Let $\Gamma$ be the path algebra of the oriented $E_6$ quiver in
Figure~\ref{fig:e6}, and set $T=\bigoplus_{i=1}^6T_i$ with the summands
marked in the Auslander--Reiten quiver in Figure~\ref{fig:are6}.
Write $R_i=e_i\Gamma$ and $S_i^\Gamma=R_i/\rad R_i$. Then
\[
T_1=S_6^\Gamma,\quad T_2=Y,\quad T_3=R_1,\quad
T_4=S_1^\Gamma,\quad T_5=R_6,\quad T_6=R_4,
\]
where $Y$ has space $k$ at every vertex and identity arrow maps.
Since $\Gamma$ is hereditary, $\operatorname{pd}_\Gamma T\le1$,
and a direct calculation gives $\Ext^1_\Gamma(T,T)=0$.
Together with the short exact sequence
\[0\to \Gamma \to T_3^3\oplus T_5^3\oplus T_6\to T_1\oplus T_2\oplus T_4\to 0,\]
this proves that the $\Gamma$-module $T$ is tilting.

\begin{figure}[htbp]
	\centering
	\begin{minipage}[c]{0.48\textwidth}
		\centering
		\begin{tikzpicture}[
			x=1.8cm,
			y=1.5cm,
			>={Stealth[length=2.2mm]},
			thick,
			every node/.style={font=\Large, inner sep=3pt},
			scale=0.8
			]
			\node (v1) at (0,0) {$1$};
			\node (v2) at (1,0) {$2$};
			\node (v3) at (2,0) {$3$};
			\node (v4) at (2,1) {$4$};
			\node (v5) at (3,0) {$5$};
			\node (v6) at (4,0) {$6$};
			
			\draw[->] (v1) -- (v2);
			\draw[->] (v2) -- (v3);
			\draw[->] (v3) -- (v4);
			\draw[->] (v5) -- (v3);
			\draw[->] (v6) -- (v5);
		\end{tikzpicture}
	\end{minipage}
	\hfill
	\begin{minipage}[c]{0.48\textwidth}
		\centering
		\begin{tikzpicture}[
			>=Stealth,
			thick,
			every node/.style={font=\Large},
			scale=0.6,
			xscale=-1
			]
			\node (1) at (0, 2) {$1$};
			\node (4) at (0,-2) {$4$};
			\node (2) at (3, 0) {$2$};
			\node (3) at (6, 2) {$3$};
			\node (5) at (6,-2) {$5$};
			\node (6) at (9, 0) {$6$};
			
			\draw[->] (1) -- node[above] {$\alpha$} (2);
			\draw[->] (4) -- node[below] {$\beta$} (2);
			\draw[->] (2) -- node[above] {$\gamma$} (3);
			\draw[->] (2) -- node[below] {$\delta$} (5);
			\draw[->] (3) -- node[above] {$\rho$} (6);
			\draw[->] (5) -- node[below] {$\sigma$} (6);
		\end{tikzpicture}
	\end{minipage}
	\caption{The quivers of $\Gamma$ and its tilted algebra $\Lambda$}\label{fig:e6}
\end{figure}
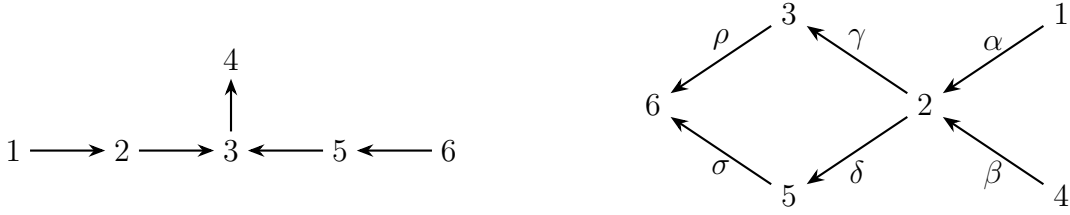

Let $\Lambda=\End_\Gamma(T)$. It is given by the right-hand quiver
in Figure~\ref{fig:e6}, with relations
\[
\alpha\gamma=\beta\delta=0,\qquad
	\gamma\rho=\delta\sigma=:c.
\]
The classical results in
\cite[Chapter~III]{Happel} yield
$\mathrm D^b(\mod \Lambda)\simeq\mathrm D^b(\mod \Gamma)$.
We claim there is a two-term silting complex $\mathbf P\in \Kb(\proj \Lambda)$ with 
\[
I_1=(\gamma),\qquad I_2=(\gamma,\delta)=I_3,
\] and thus $h_{\Lambda}(\mathbf P)=2$.

Write $P_i=e_i\Lambda$ and $S_i$ the corresponding simple $\Lambda$-modules. Take $\mathbf P$ to be the direct sum of the following summands:
\[
\begin{gathered}
	\mathbf P_{\mathrm{I}}=[P_2\xrightarrow{\alpha}P_1],\quad
		\mathbf P_{\mathrm{II}}=[P_2\xrightarrow{\beta}P_4],\quad
	\mathbf P_{\mathrm{III}}=[P_{2}\xrightarrow{(\alpha,-\beta)^t}P_{1}\oplus P_{4}],\\
	\mathbf P_{\mathrm{IV}}=P_{3}[0],\quad
		\mathbf P_{\mathrm{V}}=[P_{6}\xrightarrow{(\rho,-\sigma)^t}P_{3}\oplus P_{5}],\quad
		\mathbf P_{\mathrm{VI}}=[P_{6}\xrightarrow{\rho}P_{3}].
\end{gathered}
\]

By
\cite[Lemma~2.5, Lemma~2.2(2)]{HKM} with $n=1$, we have
\[
\Hom(\mathbf P,\mathbf P[1])\simeq\Hom(\mathbf P,H^0(\mathbf P)[1])\simeq
\operatorname{coker}\bigl(
\Hom_\Lambda(P^0,H^0(\mathbf P))\xrightarrow{(-)\circ d_{\mathbf P}}
\Hom_\Lambda(P^{-1},H^0(\mathbf P))
\bigr).
\]
It can be checked that the map $\Hom_\Lambda(P^0,H^0(\mathbf P))\xrightarrow{(-)\circ d_{\mathbf P}}
\Hom_\Lambda(P^{-1},H^0(\mathbf P))$ on each source summand of $\mathbf P$ and target
summand of $H^0(\mathbf P)$ is surjective. Therefore
$\Hom(\mathbf P,\mathbf P[1])=0$ and $\mathbf P$ is presilting.

Moreover, there is
a distinguished triangle
\[
\Lambda\longrightarrow\mathbf P_{\mathrm{III}}^{3}\oplus \mathbf P_{\mathrm{IV}}^{2}\oplus \mathbf P_{\mathrm{V}}
\longrightarrow \mathbf P_{\mathrm{I}}^{2}\oplus \mathbf P_{\mathrm{II}}^{2}\oplus \mathbf P_{\mathrm{VI}}^{2}
\longrightarrow \Lambda[1].
\]
It follows that 
$\Lambda \in\thick(\mathbf P)$. Consequently
$\thick(\mathbf P)=\Kb(\proj \Lambda)$, and $\mathbf P$ is silting.

By Lemma~\ref{lem:hom-id}, the negative chain maps correspond to
$\Hom_\Lambda(H^0(\P),H^{-1}(\P))$.
Their matrix entries generate $I_1$ by Lemma~\ref{lem:matrix-hrs-ideal}.
The three nonzero kernels are
\[
H^{-1}(\P_{\mathrm{I}})\cong P_3,\qquad
H^{-1}(\P_{\mathrm{II}})\cong P_5,\qquad
H^{-1}(\P_{\mathrm{III}})\cong S_6.
\]
The remaining three summands have zero kernel.
Comparing $H^0(\P)$ with these three modules gives
$\Hom_\Lambda(H^0(\mathbf P),H^{-1}(\mathbf P))\simeq k,$
generated by the map $H^0(\P_{\mathrm{IV}})\to H^{-1}(\P_{\mathrm{I}})$
whose corresponding negative chain map has coefficient $\gamma$.
Lemma~\ref{lem:matrix-hrs-ideal} gives
\[
I_1=(\gamma)=\langle\gamma,\beta\gamma,\gamma\rho\rangle_k.
\]

For $R_1=H^{-1}(\P)\cap P^{-1}I_1$, a direct calculation gives
\[R_1=\gamma \Lambda\oplus k(\gamma\rho)\oplus k(\gamma\rho).\]
Hence
\[H^{-1}(\mathbf P)/R_1\simeq\delta \Lambda/k(\gamma\rho)\simeq S_5.\]
By Theorem~\ref{thm:main-kernel-filtration},
$\Hom(\mathbf P_1,\mathbf P_1[-1])\simeq k$.
Its generator is represented by
\[
L=(0,\bar\delta): \P_{\mathrm{V}}^0\longrightarrow\P_{\mathrm{II}}^{-1}
\quad\text{over }\Lambda/I_1.
\]
Indeed, $\beta\delta=0$ and
$L d_{\mathbf P_{\mathrm{V}}}=-\overline{\delta\sigma}=0$
over $\Lambda/I_1$.
Since every negative chain map $\P\to\P[-1]$ over $\Lambda$ is a scalar
multiple of the map with coefficient $\gamma$, all such maps reduce to zero
modulo $I_1$.
The nonzero map $L$ therefore does not lift to a negative chain map over
$\Lambda$.
Lemma~\ref{lem:matrix-hrs-ideal} now gives
$I_2/I_1=\langle\bar\delta,\overline{\alpha\delta}\rangle_k$.
Taking its inverse image under $\Lambda\twoheadrightarrow\Lambda/I_1$
gives $I_2=(\gamma,\delta)$.

Since $e_2I_2=\langle\gamma,\delta,c\rangle_k$, each nonzero kernel lies
in the $I_2$-multiple of its corresponding $P_2$-summand.
Hence $R_2=H^{-1}(\mathbf P)$.
Theorem~\ref{thm:main-kernel-filtration} yields
$I_3=I_2$ and $h_\Lambda(\mathbf P)=2$.
Therefore
\[
0=I_0\subsetneq I_1\subsetneq I_2=I_3,
\quad h_{\mathrm{HRS}}(\Lambda)\ge h_\Lambda(\mathbf P)=2.
\]
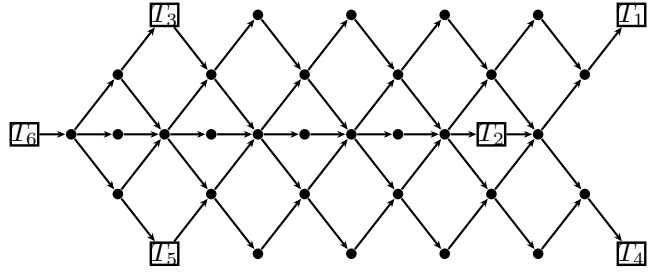
\begin{figure}
	\centering
	\begin{tikzpicture}[
		x=1.25cm,
		y=0.95cm,
		vertex/.style={circle,fill=black,inner sep=1.4pt},
		boxed/.style={rectangle,draw=black,fill=white,minimum size=6pt,
			inner sep=0pt,line width=0.9pt},
		edge/.style={-{Stealth[length=3.2pt,width=2.6pt]},thick},
		rotate=90,xscale=0.7,scale=0.5
		]
		\node[boxed]  (A)  at (0,1.3) {$T_6$};
		\node[vertex] (A0) at (0,0) {};
		
		\foreach \row/\yy in {A/-1.3,C/-3.9,E/-6.5,G/-9.1} {
			\node[vertex] (\row1) at (-1.8,\yy) {};
			\node[vertex] (\row2) at (0,\yy) {};
			\node[vertex] (\row3) at (1.8,\yy) {};
		}
		\node[boxed]  (B1) at (-3.6,-2.6) {$T_5$};
		\node[vertex] (B2) at (0,-2.6) {};
		\node[boxed]  (B3) at (3.6,-2.6) {$T_3$};
		\foreach \row/\yy in {D/-5.2,F/-7.8} {
			\node[vertex] (\row1) at (-3.6,\yy) {};
			\node[vertex] (\row2) at (0,\yy) {};
			\node[vertex] (\row3) at (3.6,\yy) {};
		}
		
		\node[vertex] (H1) at (-3.6,-10.4) {};
		\node[vertex] (H2) at (0,-10.4) {};
		\node[vertex] (H3) at (3.6,-10.4) {};
		\node[vertex] (I1) at (-1.8,-11.7) {};
		\node[boxed] (I2) at (0,-11.7) {$T_2$};
		\node[vertex] (I3) at (1.8,-11.7) {};
		\node[vertex] (J1) at (-3.6,-13.0) {};
		\node[vertex] (J0) at (0,-13.0) {};
		\node[vertex] (J2) at (3.6,-13.0) {};
		\node[vertex] (K1) at (-1.8,-14.3) {};
		\node[vertex] (K2) at (1.8,-14.3) {};
		\node[boxed]  (L1) at (-3.6,-15.6) {$T_4$};
		\node[boxed]  (L2) at (3.6,-15.6) {$T_1$};
		
		\draw[edge] (A)--(A0);
		\foreach \j in {1,2,3} {
			\draw[edge] (A0)--(A\j);
		}
		\foreach \upper/\lower in {A/B,C/D,E/F,G/H} {
			\foreach \u/\v in {1/1,1/2,2/2,3/2,3/3} {
				\draw[edge] (\upper\u)--(\lower\v);
			}
		}
		\foreach \upper/\lower in {B/C,D/E,F/G,H/I} {
			\foreach \u/\v in {1/1,2/1,2/2,2/3,3/3} {
				\draw[edge] (\upper\u)--(\lower\v);
			}
		}
		\foreach \u/\v in {1/1,1/0,2/0,3/0,3/2} {
			\draw[edge] (I\u)--(J\v);
		}
		\foreach \u/\v in {1/1,0/1,0/2,2/2} {
			\draw[edge] (J\u)--(K\v);
		}
		\foreach \j in {1,2} {
			\draw[edge] (K\j)--(L\j);
		}
	\end{tikzpicture}
	\caption{The AR quiver of $\Gamma$ and its tilting module $T$}\label{fig:are6}
\end{figure}

\end{exm}

\section{Arbitrary HRS height and delayed stabilization}
\label{sec:arbitrary-height}
We construct algebras with two simple modules and Loewy length three
whose HRS heights are arbitrarily large. The construction associates a
square-zero extension and a two-term silting complex to a nilpotent
endomorphism $\phi:V\to V$. Passing to the first HRS quotient replaces
$V$ by $V/\ker\phi$, so the kernel filtration of $\phi$ determines the
successive quotients.
\subsection{A square-zero extension and its first HRS ideal}
\label{subsec:nilpotent-operator-families}

Let $Q$ be the Kronecker quiver
\[
\begin{tikzpicture}[>=stealth,baseline=(current bounding box.center),transform shape,scale=0.7]
	\node (A) at (0,0) {0};
	\node (B) at (2,0) {1};
	\draw[->,thick] (A) to[out=30,in=150] node[sloped,above] {$\alpha$} (B);
	\draw[->,thick] (A) to[out=-30,in=-150] node[sloped,below] {$\beta$} (B);
\end{tikzpicture}
\]
and put $B=kQ$.  We use the convention $\alpha,\beta\in e_0Be_1$.

Let $V$ be a nonzero finite-dimensional $k$-vector space and let
$\phi:V\to V$ be a nilpotent linear transformation.  Define a $B$--$B$-bimodule $\mathcal M_{\phi}$ by
\[
e_1\mathcal M_{\phi}e_0=V^{\oplus2},\qquad
e_0\mathcal M_{\phi}e_0=V,\qquad
e_1\mathcal M_{\phi}e_1=V,\qquad
e_0\mathcal M_{\phi}e_1=0.
\]
The only nonzero arrow actions start from $e_1\mathcal M_{\phi}e_0=V^{\oplus2}$.
For $(\V,\w)\in V^{\oplus2}$ set
\begin{equation}\label{eq:nilpotent-bimodule-actions}
	\begin{aligned}
		L_\alpha(\V,\w):=\alpha(\V,\w)&=\w,
		&L_\beta(\V,\w):=\beta(\V,\w)&=(1+\phi)\V+\phi\w,\\
		R_\alpha(\V,\w):=(\V,\w)\alpha&=\phi(\V+\w),
		&R_\beta(\V,\w):=(\V,\w)\beta&=\phi\V+\w.
	\end{aligned}
\end{equation}
All arrow actions on the remaining corners are zero.  Thus the four corner spaces
and the left and right arrow actions may be displayed as
\[
\begin{tikzcd}[column sep=5.2em,row sep=3.7em]
	V^{\oplus2}
	\arrow[r,shift left=.8ex,"R_\alpha"]
	\arrow[r,shift right=.8ex,swap,"R_\beta"]
	\arrow[d,shift left=.8ex,"L_\alpha"]
	\arrow[d,shift right=.8ex,swap,"L_\beta"]
	& V
	\arrow[d,shift left=.8ex,"0"]
	\arrow[d,shift right=.8ex,swap,"0"] \\
	V
	\arrow[r,shift left=.8ex,"0"]
	\arrow[r,shift right=.8ex,swap,"0"]
	& 0 .
\end{tikzcd}
\]

Let $\Lambda_{\phi}:=B\ltimes\mathcal M_{\phi}$, with multiplication
\[
(b,m)(b',m')=(bb',bm'+mb').
\]
Thus $\mathcal M_{\phi}$ is a square-zero ideal of $\Lambda_{\phi}$.
Let $E_i:=e_i\Lambda_{\phi}$ denote the indecomposable projectives.
We define the two-term complex
$\P_{\phi}:=X_{\phi}\oplus Y_{\phi},$
where
\[
X_{\phi}=E_1\xrightarrow{\binom{\alpha}{\beta}}E_0^{\oplus2},
\quad
Y_{\phi}=E_1^{\oplus2}
\xrightarrow{\left(\begin{smallmatrix}\alpha&0\\ \beta&\alpha\\0&\beta\end{smallmatrix}\right)}
E_0^{\oplus3}.
\]
Write $\P_{\phi}=[P^{-1}\xrightarrow d P^0]$, where
$P^{-1}=E_1^{\oplus3}$ and $P^0=E_0^{\oplus5}$.
Its differential is
$d=\alpha A_0+\beta C_0,$
where
\[
A_0=
\begin{pmatrix}
	1&0&0\\
	0&0&0\\
	0&1&0\\
	0&0&1\\
	0&0&0
\end{pmatrix},
\qquad
C_0=
\begin{pmatrix}
	0&0&0\\
	1&0&0\\
	0&0&0\\
	0&1&0\\
	0&0&1
\end{pmatrix}.
\]
If $U\subseteq V$ is $\phi$-invariant, we write $\mathcal M_{\phi}(U)$ for the
four-corner subspace obtained from the same construction by replacing
$(V,\phi)$ with $(U,\phi|_U)$.

Since $\mathcal M_{\phi}$ is nilpotent,
$\rad\Lambda_{\phi}=\rad B\oplus\mathcal M_{\phi}$, as vector spaces, and
$\Lambda_{\phi}/\rad\Lambda_{\phi}\simeq k\times k$.
Consequently, $\Lambda_{\phi}$ is a finite-dimensional
split basic algebra with two simple modules. Moreover,
$(\rad B)^2=\mathcal M_{\phi}^2=0$, and the arrow actions give
\[
(\rad\Lambda_{\phi})^2
=e_0\mathcal M_{\phi}e_0\oplus e_1\mathcal M_{\phi}e_1.
\]
Indeed, $\alpha(0,v)=v$ and $(0,v)\beta=v$ in the two respective corners.
Since the arrows and $\mathcal M_{\phi}$ annihilate these corners on both
sides, $(\rad\Lambda_{\phi})^3=0$. Its square is nonzero for $V\ne0$,
so the Loewy length is three.  When $V=0$, we use the same notation
with $\mathcal M_{\phi}=0$, $\Lambda_{\phi}=B$, and $\P_{\phi}$ equal to the displayed
Kronecker complex; this includes the terminal quotient of the construction.

For the following computations, set $\Lambda=\Lambda_{\phi}$,
$M=\mathcal M_{\phi}$ and $\P=\P_{\phi}$.
Since $e_0\Lambda e_1=k\alpha\oplus k\beta$,
every map $P^{-1}\to P^0$ is uniquely of the
form
$\alpha U+\beta V$ with $U,V\in M_{5\times3}(k)$,
and therefore the map
\begin{equation}\label{eq:theta-definition}
	\theta:
	M_{5\times3}(k)\oplus M_{5\times3}(k)
	\xrightarrow{\sim}
	\Hom_{\Lambda}(P^{-1},P^0),
	\;
	(U,V)\mapsto \alpha U+\beta V,
\end{equation}
is a $k$-linear isomorphism.

Since $\Lambda_{\phi}=B\oplus M$ as vector spaces, we have
$e_i\Lambda_{\phi}e_i=ke_i\oplus e_iMe_i$ for $i=0,1$.
For $R\in\End_{\Lambda}(P^0)\simeq M_5(e_0\Lambda e_0)$ and
$S\in\End_{\Lambda}(P^{-1})\simeq M_3(e_1\Lambda e_1)$, write uniquely
\[
R=R_0+R_M\qquad\text{and}\qquad S=S_0+S_M.
\]
Here $R_0\in M_5(k)$ and $S_0\in M_3(k)$, while the entries of
$R_M$ and $S_M$ lie in $e_0Me_0$ and $e_1Me_1$, respectively.
The quotient map $\Lambda\to B$ induces the projections
$\End_{\Lambda}(P^j)\twoheadrightarrow\End_B(\bar P^j)$ for $j=-1,0$,
where $\bar P^j=P^j\otimes_\Lambda B$.
Write $\bar X=X_{\phi}\otimes_{\Lambda}B$ and
$\bar Y=Y_{\phi}\otimes_{\Lambda}B$.
Note that $\End_B(\bar P^0)\cong M_5(k)$ and
$\End_B(\bar P^{-1})\cong M_3(k)$.
We define the projection  
\begin{equation}\label{eq:pi-definition}
	\pi:
	\End_{\Lambda}(P^0)\oplus\End_{\Lambda}(P^{-1})
	\twoheadrightarrow M_5(k)\oplus M_3(k),
	\quad
	(R,S)\mapsto (R_0,S_0).
\end{equation}

\begin{proposition}
	\label{prop:nilpotent-silting-realization}
	The complex $\P_{\phi}$ is a two-term silting complex over $\Lambda_{\phi}$.  For its
	induced torsion pair
	$(\T(\P_{\phi}),\F(\P_{\phi})),$
	we have
	\[
	J(\P_{\phi})=\mathcal M_{\phi}.
	\]
\end{proposition}
\begin{proof}
	We divide the proof into two steps.
	
	\smallskip
	\noindent\emph{Step 1: $\P$ is a two-term silting complex.}
	Since $\P$ is two-term, to prove that it is presilting it is enough to show
	$\Hom(\P,\P[1])=0$.
	Consider the morphism
	\[
	\delta:\End_{\Lambda}(P^0)\oplus\End_{\Lambda}(P^{-1})
	\longrightarrow\Hom_{\Lambda}(P^{-1},P^0),\qquad
	(R,S)\mapsto Rd+dS.
	\]
	There is an isomorphism $\Hom(\P,\P[1]) \cong \Coker \delta$.

	Define a $k$-linear map 
	\begin{equation}\label{eq:Psi-definition}
		\Psi:
		M_5(k)\oplus M_3(k)
		\longrightarrow M_{5\times3}(k)^{\oplus2},
		\quad
		(R_0,S_0)\mapsto
		\bigl(R_0A_0+A_0S_0,\,
		R_0C_0+C_0S_0\bigr).
	\end{equation}
	Since $R_Md=dS_M=0$, the map $\delta$ depends only on $R_0,S_0$,
	and hence $\delta=\theta\circ\Psi\circ\pi$.
	
	Since $\pi$ is surjective and $\theta$ is an isomorphism, $\theta$ induces an isomorphism
	$\Coker\Psi
	\xrightarrow{\sim}
	\Coker\delta.$
	Therefore  $\Coker \delta=0 $
	is equivalent to the surjectivity of $\Psi$. 
	The equation
	$\Psi(R_0,S_0)=0$ holds if and only if
	\[
	R_0A_0+A_0S_0=0,
	\qquad
	R_0C_0+C_0S_0=0.
	\]
	After replacing $S_0$ by $-S'_0$, the solutions are precisely
	\[
	R_0=\begin{pmatrix}
	 a&0&0&0&0\\
	 0&a&0&0&0\\
	 b&0&t&0&0\\
	 c&b&0&t&0\\
	 0&c&0&0&t
	\end{pmatrix},\qquad
	S'_0=\begin{pmatrix}
	 a&0&0\\ b&t&0\\ c&0&t
	\end{pmatrix}
	\quad(a,b,c,t\in k).
	\]
	Thus $\dim_k\ker\Psi=4$.
	Since the domain and target have dimensions $34$ and $30$, respectively,
	$\Psi$ is surjective. Therefore $\Hom(\P,\P[1])=0$.

	The diagonal blocks corresponding to $\bar X$ and $\bar Y$ give
	$\End_{\Kb(\proj B)}(\bar X)=\End_{\Kb(\proj B)}(\bar Y)=k$.
	Since $e_1Be_0=0$, homotopies over $B$ vanish, and the different
	projective multiplicities show that $\bar X\not\simeq\bar Y$.
	For $Z=X_\phi,Y_\phi$, reduction modulo $M$ induces a homomorphism
	\[
	\End_{\Kb(\proj\Lambda)}(Z)
	\longrightarrow
	\End_{\Kb(\proj B)}(Z\otimes_\Lambda B)=k.
	\]
	Since homotopies over $B$ vanish, every element of the kernel has
	a representative with entries in $M$. Since $M^2=0$, the kernel is
	square-zero. Scalar identities lift, so
	$\End_{\Kb(\proj\Lambda)}(Z)$ is local.
	Hence $X_\phi,Y_\phi$ are nonisomorphic indecomposables and $\P$ is basic.
	Since $\Lambda$ has exactly two simple modules, we have
	$|\P|=|\Lambda|=2$.  By 
	\cite[Proposition~3.3(b)]{AIR}, a basic two-term presilting complex with
	$|\P|=|\Lambda|$ is silting.  Therefore $\P$ is a basic two-term silting complex.
	
	\smallskip
	\noindent\emph{Step 2: computation of the common annihilator.}
	We first prove that $M$ annihilates $\T(\P)=\Fac H^0(\P),$
	and 
	$\F(\P)=\Sub H^{-1}(\nu \P)$.
	
	For the differential
	$d:P^{-1}\longrightarrow P^0$
	of $\P$, we restrict it to a homomorphism $d_M:P^{-1}M\longrightarrow P^0M.$
	For $X_\phi$, the corner decomposition identifies its restricted differential with
	\[
	V^{\oplus2}\oplus V\xrightarrow{(F_\phi,0)}V^{\oplus2},
	\qquad F_\phi(v,w)=\bigl(w,(1+\phi)v+\phi w\bigr).
	\]
	The last summand $V=e_1Me_1$ maps to zero since $e_0Me_1=0$.
	Since $1+\phi$ is invertible, $F_\phi$ is an isomorphism, with
	\[
	F_\phi^{-1}(a,b)=\bigl((1+\phi)^{-1}(b-\phi a),a\bigr).
	\]
	Hence $H^0(X_\phi)M=0$.
	
	For
	the summand 
	$Y_\phi$ of $\P$,
	the nonzero restriction of $d_M\colon P^{-1}M\to P^0M$ is
	\[
	(e_1Me_0)^{\oplus2}\longrightarrow(e_0Me_0)^{\oplus3},
	\qquad (x,y)\longmapsto(L_\alpha x,L_\beta x+L_\alpha y,L_\beta y).
	\]
	Every $(a,b,c)$ is the image of
	$\bigl(F_\phi^{-1}(a,b),F_\phi^{-1}(0,c)\bigr)$.
	Thus the map is surjective, and
	$H^0(Y_\phi)M=0$
	and 
	$H^0(\P)M=0.$
	
	We show that
	$H^{-1}(\nu\P)M=0.$
	Let
	$\P^\ast:=\Hom_{\Lambda}(\P,\Lambda).$
	Since $\nu\P=D\P^\ast,$
	one has
	$H^{-1}(\nu\P)
	\simeq
	D\,H^1(\P^\ast).$
	The nonzero parts of the restricted differentials for $X_\phi^*$
	and $Y_\phi^*$ are
	\[
	(x,y)\longmapsto x\alpha+y\beta,
	\qquad
	(x,y,z)\longmapsto(x\alpha+y\beta,y\alpha+z\beta).
	\]
	They send $(0,(0,a))$ to $a$ and
	$(0,(0,a),(0,b-\phi a))$ to $(a,b)$, respectively.
	Thus both maps are surjective, and the left $\Lambda$-module
	$H^1(\P^*)$ is annihilated by $M$. 
	Thus we obtain
	\[
	M\subseteq
	\Ann_{\Lambda}(\T(\P))\cap\Ann_{\Lambda}(\F(\P))
	=J(\P).
	\]
	
	Conversely, let $\rho:\Lambda\twoheadrightarrow B=\Lambda/M$
	be the quotient map.  Since $M$ annihilates $H^0(X_{\phi})$, this module is
	naturally a $B$-module and identifies with $H^0(\bar X)$.
	Let $x,y$ denote the images of $(e_0,0)$ and $(0,e_0)$ in
	$H^0(\bar X)=\Coker[e_1B\xrightarrow{\binom{\alpha}{\beta}}(e_0B)^2]$.
	They generate its vertex-$0$ component. At vertex $1$, the generators
	are $x\alpha,x\beta,y\alpha,y\beta$, with the sole relation
	$x\alpha+y\beta=0$, since the differential sends $e_1$ to $(\alpha,\beta)$. If $a_0e_0+a_1e_1+b\alpha+c\beta$ annihilates it,
	its action at vertex $1$ gives $a_1=0$. Its action on $x$ gives $a_0x+b x\alpha+c x\beta=0$.
	Since $x$ lies at vertex $0$ and $x\alpha,x\beta$ at vertex $1$,
	the two components vanish separately. Since $x\ne0$ and
	$x\alpha,x\beta$ are linearly independent, $a_0=b=c=0$.
	Thus $H^0(\bar X)$ is faithful over $B$.
	Since $J(\P)$ annihilates $H^0(X_{\phi})\in\T(\P)$, its
	image $\rho(J(\P))$ annihilates the faithful $B$-module
	$H^0(\bar X)$.  Hence $\rho(J(\P))=0$, and therefore
	$J(\P)\subseteq\ker\rho=M.$
	Combining the two inclusions gives $J(\P)=M$, as required.
\end{proof}

\begin{proposition}
	\label{prop:nilpotent-first-defect}
	For the pair $(\Lambda_{\phi},\P_{\phi})$, we have
	\begin{equation}\label{eq:nilpotent-negative-hom}
		\Hom_{\Kb(\proj \Lambda_{\phi})}(\P_{\phi},\P_{\phi}[-1])
		\simeq(\ker \phi)^{2}.
	\end{equation}
	Moreover, its first HRS ideal is
	$I_1=\mathcal M_{\phi}(\ker \phi).$
\end{proposition}

\begin{proof}
	We use the matrix description of Lemma~\ref{lem:matrix-hrs-ideal}.  Since
	$P_\phi^{-1}=E_1^3$ and $P_\phi^0=E_0^5$,
	a negative map
	$L:P_\phi^0\longrightarrow P_\phi^{-1}$
	is a $3\times5$ matrix whose entries lie in
	\[
	e_1\Lambda_\phi e_0=e_1\mathcal M_\phi e_0=V^{2}.
	\]
	By \eqref{eq:matrix-defect-kernel} and
	Lemma~\ref{lem:matrix-hrs-ideal}, it is enough to compute the kernel of
	$\Theta_d$ and the ideal generated by its matrix entries.

	Solving $dL=0$ allows us to write the $j$-th column of $L$ as
	\[
	\kappa_\phi(\w_j)=\bigl(0,(\w_j,0),(\phi \w_j,-(1+\phi)\w_j)\bigr)^{\mathsf t}.
	\]
	For $\uu\in V$ the right actions give
	
	\[
	\kappa_\phi(\uu)\alpha=
	\begin{pmatrix}0\\ \phi \uu\\ -\phi \uu\end{pmatrix},
	\qquad
	\kappa_\phi(\uu)\beta=
	\begin{pmatrix}0\\ \phi \uu\\ (-1-\phi+\phi^2)\uu\end{pmatrix}.
	\]
	Thus
	\[
	\kappa_\phi(\uu)\alpha+\kappa_\phi(\uu')\beta=0
	\iff
	\phi(\uu+\uu')=0,\quad
	-\phi \uu+(-1-\phi+\phi^2)\uu'=0.
	\]
	Eliminating $\phi \uu$ gives $(1-\phi^2)\uu'=0$.  Since $\phi$ is
	nilpotent, $1-\phi^2$ is invertible; hence $\uu'=0$ and
	$\uu\in\ker\phi$.  Applying this observation to the three equations in
	$Ld=0$ yields
	$\w_2=\w_4=\w_5=0$ and $\phi\w_1=\phi\w_3=0$.
	Hence
	\[
	\ker\Theta_d
	=
	\left\{
	\begin{pmatrix}
		0&0&0&0&0\\
		(\w,0)&0&(\V,0)&0&0\\
		(0,-\w)&0&(0,-\V)&0&0
	\end{pmatrix}
	\ \middle|\ \w,\V\in\ker \phi
	\right\}
	\simeq(\ker \phi)^{2}.
	\]

	It remains to determine the coefficient ideal $I_1$ generated by $\ker\Theta_d$.  Put
	$K=\ker\phi$.  The entries of the above matrices generate
	$e_1\mathcal M_\phi(K)e_0=K^{2}.$
	Since $\alpha(0,w)=w$ and $(0,w)\beta=w$ for $w\in K$,
	these entries also generate the two diagonal corners
	\[
	e_0\mathcal M_\phi(K)e_0=K,
	\qquad
	e_1\mathcal M_\phi(K)e_1=K,
	\]
	while $e_0\mathcal M_\phi(K)e_1=0$. Thus this coefficient ideal contains $\mathcal M_{\phi}(K)$.
	Conversely, no larger subspace can be generated by the two-sided actions, since
	$\phi(K)=0$, $\mathcal M_{\phi}^2=0$, and all arrow actions on $e_i\mathcal M_{\phi}e_i$ vanish.  Thus the two-sided ideal generated by the
	matrix entries is exactly
	$\mathcal M_\phi(K).$
	By Lemma~\ref{lem:matrix-hrs-ideal}, we have 
	$I_1=\mathcal M_\phi(\ker\phi).$
\end{proof}

\subsection{Kernel filtrations and HRS height}
\label{subsec:nilpotent-kernel-filtration}

The preceding calculation determines the first HRS ideal.  To iterate it, we
record the self-similarity of the construction under invariant quotients.

\begin{lemma}
	\label{lem:nilpotent-quotient-compatibility}
	Let $U\subseteq V$ be a $\phi$-invariant subspace and let $\bar \phi$ be the induced map on
	$V/U$.  Then $\mathcal M_{\phi}(U)$ is a two-sided ideal of $\Lambda_{\phi}$ and
	\begin{equation}\label{eq:nilpotent-self-similar-quotient}
		\Lambda_{\phi}/\mathcal M_{\phi}(U)\simeq\Lambda_{\bar \phi}.
	\end{equation}
	Under this identification,
	$
	\P_{\phi}\otimes_{\Lambda_{\phi}}
	\Lambda_{\bar\phi}
	\simeq\P_{\bar \phi}.$
	Let $q_U:\Lambda_\phi\twoheadrightarrow\Lambda_{\bar\phi}$ be the
	quotient map under this identification. For every $\phi$-invariant
	subspace $W\supseteq U$, one has
	\[
	q_U(\mathcal M_\phi(W))=\mathcal M_{\bar\phi}(W/U),\qquad
	q_U^{-1}(\mathcal M_{\bar\phi}(W/U))=\mathcal M_\phi(W).
	\]
\end{lemma}

\begin{proof}
	Since the arrow actions in \eqref{eq:nilpotent-bimodule-actions} preserve $U$,
	$\mathcal M_{\phi}(U)$ is a $B$--$B$-subbimodule of $\mathcal M_{\phi}$
	and hence a two-sided ideal of $\Lambda_{\phi}$.
	Put $\overline{\mathcal M}=\mathcal M_\phi/\mathcal M_\phi(U)$. Its corners are
	\[
	e_1\overline{\mathcal M}e_0\simeq(V/U)^{\oplus2},\qquad
	e_i\overline{\mathcal M}e_i\simeq V/U\ (i=0,1),\qquad
	e_0\overline{\mathcal M}e_1=0.
	\]
	Let $\pi:V\to V/U$ be the canonical projection. Since
	$\pi\phi=\bar\phi\pi$, these identifications respect the arrow actions,
	so $\overline{\mathcal M}\simeq\mathcal M_{\bar\phi}$ as $B$--$B$-bimodules.
	It follows that
	\[
	\Lambda_\phi/\mathcal M_{\phi}(U)
	\simeq
	B\ltimes(\mathcal M_{\phi}/\mathcal M_{\phi}(U))
	\simeq
	B\ltimes\mathcal M_{\bar\phi}
	=
	\Lambda_{\bar\phi}.
	\]
	
	Since $E_i=e_i\Lambda_\phi$, one has
	$E_i\otimes_{\Lambda_\phi}	\Lambda_{\bar\phi}
	\simeq e_i	\Lambda_{\bar\phi}.$
	Moreover, the quotient map restricts to the identity on
	$B\subseteq\Lambda_\phi$, so the elements $\alpha$ and $\beta$
	remain unchanged.  Therefore
	\[
	X_\phi\otimes_{\Lambda_\phi}	\Lambda_{\bar\phi}
	\simeq
	[	e_1	\Lambda_{\bar\phi}
	\xrightarrow{\binom{\alpha}{\beta}}
	(e_0	\Lambda_{\bar\phi})^{\oplus2}]
	=
	X_{\bar\phi},
	\]
	and similarly 
	$Y_\phi\otimes_{\Lambda_\phi}\Lambda_{\bar \phi}
	\simeq Y_{\bar\phi}.$
	Thus	$\P_\phi\otimes_{\Lambda_\phi}
	\Lambda_{\bar\phi}
	\simeq
	\P_{\bar\phi}.$
	The map $q_U$ is the identity on $B$ and the canonical quotient
	map on each copy of $V$. Since $W\supseteq U$, the image and
	inverse-image formulas follow corner by corner from
	$\pi(W)=W/U$ and $\pi^{-1}(W/U)=W$.
\end{proof}

For the rest of this subsection, let $V\ne0$ and put
\[
K_m:=\ker \phi^m,\quad
\nu(\phi):=\min\{r\ge1\mid \phi^r=0\}.
\]
The kernels form the filtration
\[0=K_0\subseteq K_1\subseteq\cdots\subseteq K_{\nu(\phi)}=V.\]
\begin{theorem}
	\label{thm:nilpotent-operator-defect}
	The even HRS filtration associated with $\P_{\phi}$ is
	\begin{equation}\label{eq:nilpotent-defect-filtration}
		I_m=\mathcal M_{\phi}(K_m)
		\quad(m\ge0).
	\end{equation}
	Consequently,
	$h_{\Lambda_{\phi}}(\T(\P_{\phi}))=\nu(\phi).$
\end{theorem}

\begin{proof}
	We prove by induction that
	$I_m=\mathcal M_{\phi}(K_m)$ for every $m\ge0$.
	The case $m=0$ follows from $I_0=K_0=0$.
	Assume $I_m=\mathcal M_{\phi}(K_m)$ and let
	$\bar\phi_m$ be the endomorphism induced by $\phi$ on $V/K_m$, namely
	\[
	\bar\phi_m(v+K_m)=\phi(v)+K_m.
	\]
	Its kernel satisfies
	\[
	v+K_m\in\ker\bar\phi_m
	\iff\phi(v)\in K_m
	\iff\phi^{m+1}(v)=0,
	\]
	so $\ker\bar\phi_m=K_{m+1}/K_m$.
	By Lemma~\ref{lem:nilpotent-quotient-compatibility},
	we identify
	\[
	\Lambda_m:=\Lambda_{\phi}/I_m\simeq\Lambda_{\bar \phi_m}.
	\]
	Under this identification, Corollary~\ref{cor:iterated-hrs-base-change} gives
	$I_{m+1}/I_m
	=
	\mathfrak b_{\Lambda_m}
	\bigl(\P_\phi\otimes_{\Lambda_\phi}\Lambda_m\bigr).$
	Hence Proposition~\ref{prop:nilpotent-first-defect} yields
	$I_{m+1}/I_m
	=
	\mathcal M_{\bar\phi_m}(\ker\bar\phi_m).$
	
	Taking inverse images under $q_{K_m}:\Lambda_\phi\twoheadrightarrow\Lambda_m$
	and applying the inverse-image formula in
	Lemma~\ref{lem:nilpotent-quotient-compatibility}, we obtain
	$I_{m+1}=\mathcal M_{\phi}(K_{m+1})$, completing the induction.
	
	For $m<\nu(\phi)$, the nilpotent operator $\bar\phi_m$ acts on the
	nonzero space $V/K_m$, so $\ker\bar\phi_m\ne0$ and
	$K_m\subsetneq K_{m+1}$. For $m\ge\nu(\phi)$, one has $K_m=V$.
	The kernel filtration, and hence the HRS ideal filtration,
	first stabilizes at $\nu(\phi)$, so
	$h_{\Lambda_{\phi}}(\T(\P_{\phi}))=\nu(\phi)$.
\end{proof}

For the following dimension formulas, choose a basic representative $\P_m$
of the transported silting complex over $\Lambda_m=\Lambda_\phi/I_m$.
It has the same additive closure as the complex in the Buan--Zhou tower.
Set
\[
c_m:=\dim_k(K_{m+1}/K_m).
\]

\begin{corollary}
	\label{cor:nilpotent-profile-dimensions}
	For every $m\ge0$,
	\begin{equation}\label{eq:nilpotent-profile-dimensions}
		\dim_k(I_{m+1}/I_m)=4c_m,
		\qquad
		\dim_k\Hom_{\Kb(\proj \Lambda_m)}(\P_m,\P_m[-1])=2c_m.
	\end{equation}
	In particular, $\P_m$ first becomes tilting at $m=\nu(\phi)$.
\end{corollary}

\begin{proof}
	Theorem~\ref{thm:nilpotent-operator-defect} gives
	\[
	I_{m+1}/I_m
	\simeq
	\mathcal M_{\bar\phi_m}(\ker\bar\phi_m),
	\quad
	\ker\bar\phi_m\simeq K_{m+1}/K_m.
	\]
	The four corners $e_i\mathcal M_{\bar\phi_m}(\ker\bar\phi_m)e_j$ of $\mathcal M_{\bar\phi_m}(\ker\bar\phi_m)$ are
	\[
	(\ker\bar\phi_m)^{\oplus2},
	\quad
	\ker\bar\phi_m,
	\quad
	\ker\bar\phi_m,
	\quad
	0.
	\]
	Hence
	$\dim_k(I_{m+1}/I_m)
	=
	4\dim_k\ker\bar\phi_m
	=
	4c_m.$

	Lemma~\ref{lem:nilpotent-quotient-compatibility} and
	Corollary~\ref{cor:iterated-hrs-base-change} give
	$\add\P_m=\add\P_{\bar\phi_m}$
	under the identification
	$\Lambda_m\simeq\Lambda_{\bar\phi_m}$.
	Since both complexes are basic, they are isomorphic in
	$\Kb(\proj \Lambda_m)$.
	Proposition~\ref{prop:nilpotent-first-defect} gives
	$\Hom(\P_m,\P_m[-1])\simeq(\ker\bar\phi_m)^{\oplus2}$.
	Hence we have
	\[
	\dim_k\Hom(\P_m,\P_m[-1])=2c_m.
	\]
	Since a two-term silting complex is tilting precisely when its degree $-1$
	self-Hom vanishes, and $c_m=0$ precisely for $m\ge\nu(\phi)$,
	$\P_m$ first becomes tilting at $m=\nu(\phi)$.
\end{proof}

\subsection{Example of arbitrary HRS height}
\label{subsec:jordan-arbitrary-height}

Within this family, the ordered dimensions of the HRS ideal filtration layers are
completely determined by the Jordan type of the linear transformation $\phi$.

\begin{corollary}
	\label{cor:jordan-profile-realization}
	Let $\lambda=(\lambda_1\ge\lambda_2\ge\cdots\ge\lambda_r\ge1)$ be a
	partition, with conjugate partition $\lambda'$, where
	$\lambda'_j=\#\{i:\lambda_i\ge j\}$. There exist a
	finite-dimensional basic algebra $\Lambda_\lambda$ and a two-term
	silting complex $\P_\lambda$ whose induced torsion pair $(\T(\P_{\lambda}),\F(\P_{\lambda}))$
	satisfies
	$h_{\Lambda_\lambda}(\T(\P_\lambda))=\lambda_1,$
	and, for $0\le m<\lambda_1$,
	\begin{equation}\label{eq:jordan-profile-realization}
		\dim_k(I_{m+1}/I_m)=4\lambda'_{m+1},\quad
		\dim_k\Hom(\P_m,\P_m[-1])=2\lambda'_{m+1}.
	\end{equation}
\end{corollary}

\begin{proof}
	Take
	$V_\lambda=\bigoplus_{i=1}^r k[z]/(z^{\lambda_i})$
	and let $\phi_\lambda$ be multiplication by $z$.  Set
	$\Lambda_\lambda=\Lambda_{\phi_\lambda}$,
	$\P_\lambda=\P_{\phi_\lambda}$, and  $(\T(\P_{\lambda}),\F(\P_{\lambda}))=(\T(\P_{\phi_\lambda}),\F(\P_{\phi_\lambda}))$.
	Since a block of size $\lambda_i$ contributes $\min(m,\lambda_i)$
	to $\dim_k\ker\phi_\lambda^m$, one has $\nu(\phi_\lambda)=\lambda_1$ and
	\[
	\dim_k\frac{\ker \phi_\lambda^{m+1}}{\ker \phi_\lambda^m}
	=\#\{i\mid\lambda_i\ge m+1\}=\lambda'_{m+1}.
	\]
	Theorem~\ref{thm:nilpotent-operator-defect} and
	Corollary~\ref{cor:nilpotent-profile-dimensions} now give the stated
	height and dimensions.
\end{proof}
\begin{exm}
	For a single Jordan block, set
	$V_N=k[z]/(z^N)$ and let $\phi_N$ be multiplication by $z$.
	Write
	\[
	\mathcal M_N:=\mathcal M_{\phi_N},\quad
	\Lambda_N:=\Lambda_{\phi_N},\quad
	\P_N:=\P_{\phi_N},
	\]
	and 
	$(\T_N,\F_N)$ for the corresponding torsion pair.

	Since $\ker \phi_N^m=z^{N-m}V_N$ for $0\le m\le N$, Theorem~\ref{thm:nilpotent-operator-defect}
	gives the explicit filtration
	\begin{equation}\label{eq:single-block-defect-chain}
		I_m^{(N)}=\mathcal M_{\phi_N}(z^{N-m}V_N)
		\qquad(0\le m\le N).
	\end{equation}
	The inclusions $I^{(N)}_0\subsetneq I^{(N)}_1\subsetneq\ldots \subsetneq I^{(N)}_N$ are strict and
	$I_N^{(N)}=I_{N+1}^{(N)}=\cdots=\mathcal M_N$.  Moreover, $\dim_k\Lambda_N=\dim_k B+\dim_k\mathcal M_N=4+4N$.
	
	For every $N\ge1$, the single-block family satisfies
	$h_{\Lambda_N}(\T_N)=N$.
	Since $\dim_k(K_{m+1}/K_m)=1$ for $0\le m<N$,
	Corollary~\ref{cor:nilpotent-profile-dimensions} gives
	\[
	\dim_k(I_{m+1}^{(N)}/I_m^{(N)})=4,
	\quad
	\dim_k\Hom(\P_m,\P_m[-1])=2.
	\]
	Thus the HRS height is unbounded even when both dimensions remain fixed
	at every stage before stabilization.
\end{exm}

\begin{remark}\label{cor:minimal-size-height-realization}
The examples use the least possible number of simple modules.
For height at least two, they also have the least possible Loewy length.

An algebra with one simple module is Morita equivalent to a local
algebra $A$.  Since Morita equivalence transports the HRS iteration,
it preserves height. Thus it suffices to work over $A$.
Represent a two-term silting complex by a minimal complex
$\P=[A^u\xrightarrow d A^v]$; its differential has all entries
in $\rad A$.
Suppose $u,v>0$, and choose a map $A^u\to A^v$ with one entry
$1_A$ and all other entries zero.  This defines a chain map
$\P\to\P[1]$.  Since every entry of a null-homotopic map
$Rd+dS$ belongs to $\rad A$, the chosen map is not null-homotopic.
This contradicts $\Hom(\P,\P[1])=0$.  Hence $u=0$ or $v=0$,
and $\P$ is concentrated in one degree.  It is therefore tilting
and has height zero.  Positive height requires at least two
simple modules.

For the second assertion, an algebra of Loewy length at most two
has radical square zero.
Corollary~\ref{cor:collapse} then gives height at most
one.  Hence height at least two requires Loewy length at least three.
The algebras $\Lambda_N$ attain both bounds.
\end{remark}
\section*{Acknowledgements}
This work was supported by the National Natural Science Foundation of
China (Grant Nos.~12001164, 12571039 and 1250011863). The authors thank
Yu Zhou for inspiring discussions.

\section*{Declaration of AI assistance}
The authors used ChatGPT (OpenAI) to assist with the development of proof
arguments, exploratory calculations, and revision of the manuscript.
The research questions and initial approaches were formulated by the
authors. The authors take full responsibility for the correctness,
originality, and presentation of the work.

\end{document}